\documentclass[9pt]{article}
\usepackage{preamble}

\providetoggle{full}
\settoggle{full}{true}

\begin{document}

\title{On the Structure of Frames and Equiangular Lines over Finite Fields and their Connections to Design Theory}

\author{Ian Jorquera\thanks{Corresponding author: \texttt{ian.jorquera@colostate.edu}} \thanks{Department of Mathematics, Colorado State University, Fort Collins, CO} \and Emily J.\ King\thanks{Department of Mathematics, Colorado State University, Fort Collins, CO}}

\maketitle

\begin{abstract}
This paper concerns frames and equiangular lines over finite fields. We find a necessary and sufficient condition for systems of equiangular lines over finite fields to be equiangular tight frames (ETFs). As is the case over subfields of $\mathbb{C}$, it is necessary for the Welch bound to be saturated, but there is an additional condition required involving sums of triple products. We also prove that similar to the case over $\mathbb{C}$, collections of vectors are similar to a regular simplex essentially when the triple products of their scalar products satisfy a certain property. Finally, we investigate switching equivalence classes of frames and systems of lines focusing on systems of equiangular lines in finite orthogonal geometries with maximal incoherent sets, drawing connections to combinatorial design theory.
\end{abstract}

\section{Introduction}
An important problem in fields as diverse as compressed sensing~\cite{BFMW13}, digital fingerprinting~\cite{MQKF13}, quantum state tomography~\cite{RBkSC04,FHS17}, multiple description coding~\cite{StH03,MeDa14,Welch}, and discrete geometry~\cite{Toth65} is to find arrangements of lines though the origin in $d$-dimensional space that are maximally spread apart, i.e., collections of lines whose smallest pairwise included angle is as large as possible. In certain cases, like when the lines are associated to an equiangular tight frame, the pairwise interior angle of such configurations is constant.

The classical theory focuses on lines in $\F^d$ where $\F$ is $\R$ or $\C$. In these cases a line $\ell$ can be represented by a vector $\phi$ which spans the line $\ell$. The non-obtuse angle $\theta$ between two lines $\ell_1$ and $\ell_2$ which are represented by the non-zero vectors $\phi_1$ and $\phi_2$ may be computed using the magnitude squared of the standard dot product $\absip{\phi_1}{\phi_2}^2=\norm{\phi_1}^2\norm{\phi_2}^2\cos^2(\theta)$. So, a system of lines $(\ell_j)_{j=1}^n$ in $\F^d$ with equal-norm representatives $(\phi_j)_{j=1}^n$ are equiangular if and only if there exists some $b\in\F$ such that $\absip{\phi_j}{\phi_k}^2=b$ for all distinct $j,k$. Equiangular systems of lines over $\R^d$ have been shown to have many connections to different objects in combinatorial designs theory, such as strongly regular graphs and two-graphs \cite{LEMMENS1973494, VaSei66}. Since then much work has been done to understand these connections leading to advancements in both fields \cite{taylortwographs, FICKUS201854}.  Combinatorial methods have also been used to characterize (e.g., switching equivalence \cite{chien_characterization_2016}) and construct equiangular systems in $\bC^d$ (e.g., difference sets~\cite{XZG05} and Steiner systems~\cite{FMT12}).

A bound on the number $n$ of equiangular lines which can be packed in $\F^d$ is know as Gerzon's bound (cf.~\cite{LEMMENS1973494}): $n\leq d(d+1)/2$ for $\F=\R$ and $n\leq d^2$ for $\bF=\bC$. In the real case Gerzon's bound is only known to be saturated when $d=2,3,7$ or $23$, and it is conjectured by Gillespie, that these are the only dimensions which saturate Gerzon's bound \cite{gillespie-2018-equiangular}. In a finite field analog of real vector spaces, an equivalent bound is saturated in infinity many dimensions \cite{greaves_frames_2021-1} but only for fields with small characteristic. Also in contrast to the real case, over $\bC$ it is conjectured (Zauner's conjecture) that $d^2$ equiangular lines exist in $\bC^d$ for all $d$~\cite{Zauner1999,Zauner2011}. Currently no infinite family of equiangular lines is known; $d^2$ equiangular lines in $\bC^d$ have been explicitly constructed in all dimensions $d \leq 53$ with sporadic further dimensions up to $5,799$ (as of March 2025 \cite{appleby2025constructive}).

A second important bound, called the relative bound, is on the maximum number of equiangular lines with respect to a fixed choice of the parameter $b\in [0,\frac{1}{d})$: A system of $n$ equiangular lines with $\absip{\phi_j}{\phi_k}^2=b$ exists in $\F^d$ for $\F=\R$ or $\C$ only if $n\leq \frac{d(1-b)}{1-bd}$. Conditions on equality may be given in terms of frame theory. Let $(\ell_j)_{j=1}^n$ be an equiangular system of lines in $\F^d$, with equal-norm representatives $(\phi_j)_{j=1}^n$. We can construct a matrix where each vector representative is the column of the $d\times n$ matrix $\Phi=\begin{bmatrix}\phi_1&\dots&\phi_n\end{bmatrix}$. A collection of lines in $\R^d$ or $\C^d$ saturates the relative bound if and only if $\Phi$ is a tight frame, meaning $\Phi\Phi^*=cI$, and we call $\Phi$ an equiangular tight frame (ETF). 
Frames have been of particular interest due to their many applications to a variety of fields ranging from wavelets~\cite{MR836025}, SICs in quantum information theory~\cite{RBkSC04}, and compressed sensing~\cite{BFMW13}.
Due to the importance of ETFs much of the work in finite dimensional frame theory is on the existence and construction of ETFs of certain sizes, and many constructions are combinatorial in nature.

This paper builds off previous work on frames over finite fields from \cite{greaves_frames_2021,greaves_frames_2021-1,iverson_note_2021}. Other authors, such as \cite{bodmann_frame_2009}, have also investigated frames for binary vector spaces under the Hamming distance. Frames over finite fields have been shown to have many structural overlaps with frames over $\R$, and $\C$, in addition to vastly different behaviors, including the existence of infinite families of frames which saturate Gerzon's bound. A main goal of exploring frame theory over more generalized fields is to find structural connections to make progress towards the resolution of conjectures, like the Gillespie and Zauner Conjectures, concerning frames over fields of characteristic zero.  However, throughout the paper, we also present a number of explicit examples (\cref{ex:ranknotiff,ex:twonotswitch,ex:discrisweird,ex:naiweird,ex:welchweird,ex:weridsimplices,ex:incohind,ex:weridsimplices2}) of vectors in over finite fields which behave in ways counter intuitive to researchers accustomed to working over characteristic zero. 

In \cref{sec:arbitraryfields,sec:frames,ssec:lines} we give an overview of frame theory over arbitrary fields using the weaker notion of Hermitian scalar products instead of inner products used in frame theory over Hilbert spaces. \cref{sec:switching} looks at an equivalence of systems of equiangular lines, known as switching equivalence, which strongly corresponds to what is known in the complex setting, except with the addition of isotropic vectors, vectors that behave like zero with respect to the scalar product. \cref{sec:combinatorics-two-graphs} then presents some of the combinatorial connections to frames in orthogonal geometries, the finite field analog to real vector spaces.

\cref{sec:beigntight} focuses on properties of tight frames and investigates the situations in which equiangular systems of lines are equiangular tight frames. In \cref{thm:theresult}, we provide an additional necessary and sufficient condition such that saturating the Welch bound means that a set of equiangular lines forms an ETF, and in \cref{lem:etfiffregulartwograph} we prove a corollary to Theorem 4.3 in \cite{greaves_frames_2021-1} where we show a connection between ETFs in orthogonal geometries to regular two-graphs.

In the last two sections we investigate the situations in which ETFs have subsets of their vectors whose triple products are all equal. In \cref{sec:simplex} we show in \cref{thm:iffsimplex} that just as in the real and complex settings, equal triple products can determine the existence of regular simplices. Likewise such sets can be considered to be incoherent sets of regular two-graphs which result in quasi-symmetric 2-designs, and in the case where Gerzon's bound is saturated, 4-designs.

\section{Frames over Arbitrary Fields and Combinatorial Designs}
\subsection{Foundations Remixed}\label{sec:arbitraryfields}
Throughout this paper we are motivated by questions in frame theory over arbitrary fields, but we
will pay special attention to finite fields.
Let $\F$ be a field with a field involution 
$(-)^\sigma:\F\ra \F$. If $\F$ is finite we will denote the field by the number 
of elements: $\F_q$ being the field with $q=p^\ell$ elements, with $p$ a prime, up to isomorphism.
If $\F=\F_{q^2}$ for $q$ a prime power, then the only involutions are the 
identity and the Frobenius involution defined as $a^\sigma=a^q$. 
In all other finite fields, the only involution is the identity.
A more thorough background on finite fields can be found in \cite{lidl_finite_1996}.

Define $\F^\times$ to denote the invertible, i.e., non-zero, elements of $\bF$ and $\F_0=\set{a\in \F}{a^\sigma=a}$ to be the elements fixed by the involution, 
which play the analogous role that $\R$ plays to $\C$, as $\C_0=\R$ with respect 
to complex conjugation. 
An element $a\in\F$ is a \textbf{square} or quadratic residue if there
exists some $x\in\F$ such that $x^\sigma x=a$
and we will denote the set of non-zero squares as $\F^{\times 2}$.
Notice that $\F^{\times 2}$ is a subgroup of the multiplicative elements fixed by the involution $\F_0^\times$ which is itself a subgroup of the multiplicative group $\F^\times$.

We give a brief overview on forms which generalize the inner product to vector spaces over finite fields, 
specifically on Hermitian scalar products, following the terminology 
from \cite{jacobson_bilinear_1953,WILSON20092642}. 
We will note that the previous papers on frames over finite fields 
\cite{greaves_frames_2021,greaves_frames_2021-1,bodmann_frame_2009} follow a slightly different naming 
convention. What we will call a Hermitian scalar product may also be referred to as 
a $\sigma$-sesquilinear form or a Hermitian form.

Throughout we assume that $V$ is a finite dimensional vector space with a \textbf{non-degenerate 
Hermitian scalar product}, which is a scalar product 
$\ip{-}{-}_V:V\times V\ra \F$, that satisfies for any $u,v\in V$
\begin{itemize}
\setlength\itemsep{0em}
\item[] (HS1) $\ip{u}{-}_V:V\ra \F$ is linear;
\item[] (HS2) $\ip{u}{v}_V=\ip{v}{u}_V^\sigma$; and
\item[] (ND)\; $u=0$ if and only if $\ip{u}{w}_V=0$ for all $w\in W$ (non-degeneracy). 
\end{itemize}

Often the underlying space $V$ will be clear from context, and so we will drop the labeling in 
the scalar product notation, and denote $\ip{-}{-}_V$ as $\ip{-}{-}$.

Notice that conditions (HS1) and (HS2) also imply additivity in the first term, 
and scalar multiples in the first term satisfy $\ip{ku}{v}=k^\sigma\ip{u}{v}$. 
These two conditions make
$\ip{-}{-}:V\times V\ra \F$ a Hermitian scalar product. If the field involution 
is the identity, then we say the Hermitian scalar product is \textbf{symmetric}. 

A subspace $W\leq V$ 
with the same Hermitian scalar product 
restricted to $W$ may not satisfy the condition of degeneracy (ND). In this case we define the \textbf{radical}
\[\rad W = \{x\in W| \ip{x,y}=0\text{ for all }y\in W\}\se W\]
which quantifies how degenerate the subspace is. Elements in the radical are called \textbf{isotropic}.
If for $u,v\in V$ we have that $\ip{u,v}=0$ we say that $u$ and $v$ are \textbf{orthogonal}, 
and we define the \textbf{orthogonal compliment} of a subspace $W$ with respect to the space $V$ 
as $W^\perp$ to be
\[W^\perp =\{v\in V| \ip{w,v}=0\text{ for all }w\in W\}\leq V.\]
This gives an equivalent definition for the radical as 
$\rad W=W^\perp\cap W$ and an equivalent condition on subspaces satisfying condition (ND).
We say a subspace $W$ is \textbf{isotropic} (or \textbf{degenerate}) 
if $\rad W=W^\perp\cap W\neq \{0\}$ and likewise
we say $W$ is \textbf{non-isotropic} (or \textbf{non-degenerate}) 
if $\rad W=W^\perp\cap W=\{0\}$, in 
which case $W$ satisfies condition (ND). A space is called \textbf{totally isotropic} if $W\leq W^\perp$.
If the scalar product for $V$ is non-degenerate then the orthogonal compliment 
is an involution, that is $W^{\perp\perp}=W$, 
but in general it would be the case that $W\se W^{\perp\perp}$. If $W$ is totally 
isotropic then $\dim W\leq \frac{1}{2} \dim V$.
Notice that if a subspace $W\leq V$ is non-isotropic then $W\cap W^\perp = \{0\}$ and $W \cup W^\perp= V$,
meaning $V= W\oplus W^\perp$.

Now assume $U,V$ are both vector spaces with non-degenerate Hermitian scalar products. 
For any linear map $A:U\ra V$ there is a unique \textbf{adjoint} which satisfies 
$\ip{Au}{v}_V=\ip{u}{A^\dagger v}_U$ for all $u\in U$ and $v\in V$. Because adjoints 
are uniquely defined, the adjoint of $A^\dagger$ is $A$, i.e.,
$(A^\dagger)^\dagger=A^{\dagger\dagger}=A$. 

For linear maps on spaces with degenerate Hermitian scalar products, 
adjoints may still be defined but are not unique and are usually defined as pairs of linear maps.
It is important to note that the adjoint is defined with respect to the non-degenerate
scalars products for $U$ and $V$ and are often not the same as the standard transpose 
or conjugate transpose which we will define separately as $A^*$, where for
$A=[a_{ij}]_{ij}\in \F^{n\times m}$ the \textbf{conjugate transpose} is $A^*=[a_{ji}^\sigma]_{ij}\in \F^{m\times n}$.

Because $V$ is finite dimensional we can give an equivalent definition of a
Hermitian scalar product and non-degeneracy in terms of the \textbf{Gram matrix}, 
$M=[\ip{e_i}{e_j}]_{ij}\in\F^{d\times d}$ where $e_1,\dots,e_d$ is a basis for $V$.
In this case, the map $\ip{-}{-}:V\times V\ra \F$, 
is a non-degenerate Hermitian scalar product for the vector space $V$ if there exists a matrix $M$ such that for
all $u,v\in V$
\begin{itemize}
\setlength\itemsep{0em}
\item[] (HS1$'$) $\ip{u}{v}=u^*Mv$;
\item[] (HS2$'$) $M=M^*$; and
\item[] (ND$'$)\; $M$ is invertible (non-degeneracy). 
\end{itemize}

This allows us to explicitly write the adjoints for linear transformations.
Let $U,V$ be vector spaces with non-degenerate Hermitian scalar products.
Choose bases $u_1,\dots, u_n$ and $v_1,\dots,v_m$ respectively, and let $M,N$ be the 
invertible matrices such that $\ip{u_1}{u_2}_U=u_1^*Mu_2$ and $\ip{v_1}{v_2}_U=v_1^*Nv_2$
Then for any map $A:U\ra V$ the unique adjoint can be written as $A^\dagger = M^{-1}A^*N$.
Two scalar products, $\ip{u_1}{u_2}_V=u_1^*Mu_2$ and $\ip{v_1}{v_2}_U=v_1^*Nv_2$ are said to be \textbf{equivalent} if there exists an invertible matrix $A:U\ra V$ such that 
$M=A^{-1}NA$.

Throughout this paper we will pay particular attention to the case where $\F$ is a finite field, in which case there are, in general, two cases for the fields. The first case is \textbf{Case U}, where $\F=\F_{q^2}$ with field involution the Frobenius involution $a^\sigma=a^q$, the unique non-trivial involution. 
In this case any vector space $V$ with a non-degenerate Hermitian scalar product is called a \textbf{unitary space}, and has \textbf{unitary geometry}. Furthermore, if $V=\F_{q^2}^d$ with the scalar product $\ip{u}{v}=u^*v$, then we say $V$ is in the \textbf{complex model}.
\begin{rem}
    In case U, where $\F=\F_{q^2}$ for a prime power $q$, it is the case that the non-zero elements fixed by the involution are exactly the non-zero squares with respect to the involution, $\F_0^\times=\F^{\times 2}$.
    To see this let $w\in\F_{q^2}^\times$ be a generator of the multiplicative group, and consider any non-zero element fixed by the involution $b\in \F_0$. This means that $b^q=b$ so $b^{q-1}=1$. Because $w$ generates $\F^\times$ there exists some $k$ such that $w^k=b$, therefore $(w^k)^{q-1}=w^{k(q-1)}=1$ so $q^2-1|k(q-1)$ and there exists a positive integer $\ell$ where $k(q-1)=\ell(q^2-1)$. This tells us that
    $k=\ell(q+1)$ and so $b=w^{\ell(q+1)}=(w^\ell)^q(w^\ell)\in \F_0^{\times}$.
\end{rem}

The second case, \textbf{Case O}, is where $\F=\F_q$, such that $q$ is an odd prime power with trivial field involution $a^\sigma=a$. In this case any vector space $V$ with a non-degenerate symmetric scalar product is called an \textbf{orthogonal space}, and has \textbf{orthogonal geometry}.
In this case $\F_0=\F$ and only half of the non-zero elements are squares. 
Furthermore if $V=\F_{q}^d$ with the scalar product $\ip{u}{v}=u^\intercal v$ then we say $V$ is in the \textbf{real model}

\begin{lem}
    \label{lem:scalarprods-diag}
    If $\ip{-}{-}$ is a non-degenerate Hermitian scalar product for $V$, then there exists a basis $v_1,\dots,v_d$ where $\ip{v_j}{v_j}=b_j$, for $b_j\in\F_0^\times$ and all other products between basis elements are zero.
    Furthermore, in Case U, it can be made so that $b_j=1$ for all $j$.
    And in case O, it can be made so that $b_j=1$ for all $j<d$ and $b_d=\delta\in \F_0^\times$ is either a square or a non-square.
\end{lem}

This tells us that we may always assume that the Gram matrix $M$ for a non-degenerate scalar product is diagonal with respect to some basis. The specific result for case O also suggests an important invariant of scalar products, the \textbf{discriminant}, which will be defined on a non-isotropic space $V$ as
\[\discr(V)=\det(M)\F^{\times 2}\in \F_0^\times/\F^{\times 2}\]
The discriminant is invariant under the choice of basis for $V$, but under a choice of basis for $V$ such as in \cref{lem:scalarprods-diag} then $\discr(V)=\delta\F^{\times 2}$. If $\discr (V)=\F^{\times 2}$, i.e., the determinant is a square, we say the discriminant is trivial. In case U, the discriminant is always trivial.

\begin{ex}
    Consider the following two matrices which can be interpreted as the Gram matrices of symmetric scalar products for $\F_3^2$, both of which have trivial discriminant and define equivalent scalar products
    \[M=\begin{bmatrix}
        1&0\\0&1
    \end{bmatrix}\;\;\;\;\;N=\begin{bmatrix}
        2&0\\0&2
    \end{bmatrix}\]
    where $2\equiv -1\pmod 3$ is not a square. 
\end{ex}
In general, two non-degenerate scalar products are equivalent if and only if they have the same discriminant.

We consider integers to be elements of $\F$. Formally, we consider an action $\cdot:\Z\times\F\ra\F$ where $n\cdot r = r+\dots +r$ added $n$-times, which defines a ring homomorphism $\Z\ra\F$ by $n\mapsto (n\cdot 1)$. More generally we can consider a surjective homomorphism from certain rings of algebraic integers $R\se \C$ to our field, $\pi:R\ra \F$ as defining an action on $\F$.
For example, let $\F_{25}=\F_5[x]/(x^2-3)=\F_5[\alpha]$, where $\alpha^2=3$. We can consider the ring $R=\Z[\sqrt{3}]$ and the map $\pi: R\ra \F_{25}$ by $\sqrt{3}\mapsto \alpha$. We often conflate algebraic integers, as elements of $\C$, with their images in $\F$. We will use $\equiv$ to represent equality in the image of $\pi$ if it not clear from context, or $\equiv_p$ if $\F=\F_{p^\ell}$ is a finite field. Often equality will be clear from context.
For $m\in R$ where $\pi(m)\neq 0$ we can consider the action of elements from the field of fractions of $R$ as
$\frac{n}{m}\mapsto (n\cdot 1)(m\cdot 1)^{-1}$. Finally, let $M$ be a matrix with entries in $R$, we will consider the $\overline{M}$ to be the pointwise image under $\pi$.

\subsection{Frames}
\label{sec:frames}
Here we continue with the traditional abuse of notation (e.g., in \cite{greaves_frames_2021,greaves_frames_2021-1,iverson_note_2021}):
identifying sequences of vectors $\Phi=(\phi_j)_{j=1}^n\se V$ with their \textbf{synthesis operator}s
$\Phi:\F^n\ra V$ defined as $\Phi(x)=\sum_{j=1}^n x_j\phi_j$. We will always assume 
any sequence of vectors lives in a finite dimensional vector space $V$ with a non-degenerate Hermitian 
scalar product, but it may be the case that the scalar product restricted to $\image\Phi\se V$ is degenerate.

The \textbf{analysis operator} $\Phi^\dagger:V\ra \F^n$ is the adjoint of the synthesis 
operator $\Phi$, where $\Phi^\dagger x=(\ip{\phi_j}{x})_j$. Its \textbf{frame operator} is defined to 
be $\Phi\Phi^\dagger$ where $x\mapsto \sum_{j=1}^n{\ip{\phi_j}{x}\phi_j}$
and its \textbf{Gram matrix} is $\Phi^\dagger\Phi=[\ip{\phi_{j}}{\phi_k}]_{jk}$.

We are interested in sequences of vectors which have additional structure. We call 
a sequence of vectors $\Phi$ a \textbf{frame} for $V$ if its vectors span $V$. We then call a frame 
$\Phi$ \textbf{non-degenerate} if its frame operator is invertible. 
If a frame $\Phi:\F^n\ra V$ satisfies $\Phi\Phi^\dagger = cI$ for some $c\in \F$,
we say $\Phi$ is a \textbf{$c$-tight frame}. In the special case when $c=0$ we call $\Phi$
a \textbf{totally isotropic tight frame}. Note that all frames in real or complex vector spaces are non-degenerate and none are totally isotropic.
Geometrically, tightness is a generalization of the Pythagorean theorem, in the sense that when
$\Phi\Phi^\dagger=cI$ for any $x\in V$, right multiplication by the frame operator gives
$cx=\sum_{j=1}^n{\ip{\phi_j}{x}\phi_j}$ and so 
$c\ip{x,x}=\sum_{j=1}^n\ip{\phi_j}{x}\ip{x}{\phi_j}$. This is exactly the Pythagorean theorem 
when the vectors of $\Phi$ are unit norm and an orthogonal basis, in which case $c=1$.
The condition for a frame to be tight can be expressed in many different ways, 
a few of which we highlight below.  
\begin{lem}
    \label{lem:tfaetight1} (Proposition 3.5 \cite{greaves_frames_2021})
    Let $\Phi:\F^n\ra V$ be a frame, then the following are equivalent:
    \begin{itemize}
        \setlength\itemsep{0em}
        \item[(i)] $\Phi$ is a $c$-tight frame meaning $\Phi\Phi^\dagger=cI$;
        \item[(ii)] $\Phi\Phi^\dagger\Phi=c\Phi$; and
        \item[(iii)] $(\Phi^\dagger\Phi)^2=c\Phi^\dagger\Phi$.
    \end{itemize}
\end{lem}

Often in this paper we will consider collections of vectors $\Phi=(\phi_j)_{j=1}^n\se V$ 
that do not span the non-isotropic space $V$ in which they live. In these cases we can still consider them to be \textbf{frames for their spans}, if $\image\Phi\se V$ is a non-isotropic subspace, meaning the Hermitian form 
for $V$ is non-degenerate when restricted to $\image\Phi$. The following lemmas can help 
characterize this situation.

\begin{lem}(Proposition 3.11 \cite{greaves_frames_2021})
    \label{lem:kerwhenframe}
    If $\Phi$ is a frame then $\ker\Phi^\dagger \Phi = \ker\Phi$ and 
    $\image\Phi^\dagger\Phi=\image\Phi^\dagger$
\end{lem}

\begin{lem}\label{lem:nondegiff}
    Let $\Phi:\F^n\ra V$ be the synthesis operator for a collection of vectors. 
    Then $\rk(\Phi^\dagger \Phi)=\rk(\Phi)$  
    if and only if $\image\Phi$ is non-isotropic, i.e., $\Phi:\F^n\ra \image \Phi$ is a frame.
\end{lem}
\begin{proof}
    The backwards direction follows from \cref{lem:kerwhenframe}.
    
    For the other direction, assume that $\rk(\Phi^\dagger \Phi)=\rk(\Phi)=\dim(\image\Phi)$. This then means $\ker(\Phi^\dagger \Phi)=\ker(\Phi)$.
    We will show that $\rad \image(\Phi) = \{u\in \image(\Phi)| \ip{u}{v}=0\text{ for all }v\in \image(\Phi)\} = \{0\}$

    Let $u\in \rad \image(\Phi)\se \image(\Phi)$. We can then write $u=\Phi x$ in which case 
    $\Phi^\dagger \Phi x = 0$, and then since $\ker(\Phi^\dagger \Phi)=\ker(\Phi)$, $x\in \ker\Phi$, i.e., $u=0$.
\end{proof}

\begin{lem}(Corollary 3.9 of \cite{greaves_frames_2021})
    \label{lem:totally_iso_ngeq2d}
    $\Phi$ is a totally isotropic tight frame if and only if $\im\Phi^\dagger$ is a totally isotropic subspace of $\F^n$. Furthermore if $V$ is a $d$-dimensional non-isotropic space then $V$ admits a $0$-tight frame 
    with $n$ vectors only if $n\geq 2d$.
\end{lem}

\subsection{Equiangular Lines and ETFs}
\label{ssec:lines}
Frames are often used as algebraic tools for studying systems of lines, that is
packings in projective space.
In the real and complex setting it is often advantageous to represent 
a line through the origin by a vector, 
specifically by a unit vector. But this is often not possible 
in finite fields, due to the lack of square roots, making it difficult to scale vectors.
Instead we will often work with the generalization of equal-norm vectors (allowing zero for the ``norm'') where we refer to the \textbf{magnitude} or \textbf{norm} of a vector $\phi$ as $\ip{\phi}{\phi}$. 
But even this is not
always possible for any given system of lines.

In general for a system of lines $\Phi=(\phi_j)_{j=1}^n\in V$, if $\frac{\ip{\phi_1}{\phi_1}}{\ip{\phi_j}{\phi_j}}$ is a square 
for every vector $\phi_j\in \Phi$ (i.e. there exists some non-zero $\alpha_j$ such that 
$\frac{\ip{\phi_1}{\phi_1}}{\ip{\phi_j}{\phi_j}}=\alpha_j\alpha_j^\sigma$) then we 
could rescale each vector to get a equal norm system of vectors $\{\alpha_j\phi_j\}_{j=1}^n$ who all share 
the common magnitude equal to the magnitude of $\phi_1$.
Furthermore to rescale an equal norm system of vectors $\Phi=(\phi_j)\in V$ 
with vectors of magnitude $a\neq 0$ to vectors of unit norm, it would need to be 
the case that $a=\alpha\alpha^\sigma$ for some non-zero $\alpha$, in which case
$\left\{\frac{1}{\alpha}\phi_j\right\}_{j=1}^n$ would be a unit norm system of vectors.
We note, however, that $a$ may often be $0$, and so scaling $a$ is frequently not possible.

To study systems of lines, specifically the ``distances'' between lines induced by the scalar products, we 
can limit ourselves to looking at equal norm systems of lines which satisfy a condition analogous to that 
of the modulus squared of the inner product of two unit vectors.

\begin{defn}
    \label{def:equisys}
    Given $a,b\in \F_0$ we say $\Phi=(\phi_j)_{j=1}^n\se V$ is an 
    \textbf{$(a,b)$-equiangular system} in $V$ if the following two conditions hold.
    \begin{itemize}
        \setlength\itemsep{0em}
        \item[(i)] $\ip{\phi_j}{\phi_j}=a$ for all $j$
        \item[(ii)] $\ip{\phi_j}{\phi_k}\ip{\phi_k}{\phi_j}=b$ for all $j\neq k$
    \end{itemize}
\end{defn}

Instead of scaling by $a$ it is often the case that $b\neq 0$; so, we may 
want to scale a system of lines such that $b=1$. Let $\Phi=(\phi_j)_{j=1}^{n}$ 
be an $(a,b)$-equiangular system of lines. If there exists a non-zero $\alpha\in \F$ such that
$b=\frac{1}{(\alpha\alpha^\sigma)^2}$ then $\{\alpha\phi_j\}_{j=1}^{n}$ would be an 
$(\alpha\alpha^\sigma a, 1)$-equiangular system of lines.

\begin{defn}
    Let $\Phi=(\phi_j)_{j=1}^n$ be a collection of vectors in a $d$-dimensional non-isotropic space $V$.
    We call $\Phi$ an \textbf{$(a,b,c)$-equiangular tight frame}, or an \textbf{$(a,b,c)$-ETF} for $V$ if (i) $\Phi$ is a 
    $(a,b)$-equiangular system of vectors and (ii) $\Phi$ is an $c$-tight frame for $V$.
\end{defn}

\section{Switching Equivalence}\label{sec:switching}
In studying frames it is helpful to create a notion of equivalence. More generally 
we will introduce a notion of equivalence for systems of lines, inspired by the switching equivalence of graphs~\cite{VaSei66}. 

\begin{defn}
    Let $\Phi =( \phi_j)_{j=1}^n, \Psi = ( \psi_j)_{j=1}^n  \se V$ where $V$ is a non-isotropic 
    $d$-dimensional space.
    We say that $\Phi$ and $\Psi$ are \textbf{unitarily equivalent} if there exists a 
    unitary $U:V\ra V$ such that  $\Psi = U \Phi$. 
    
    More generally, we say that $\Phi$ and $\Psi$ are \textbf{switching equivalent} if 
    there exists a unitary $U:V\ra V$ 
    and diagonal matrix $T=\diag(t_1,\dots, t_n)\in \F^{n \times n}$ with entries satisfying 
    $t_{i} t_{i}^\sigma=1$ such that  $\Psi = U \Phi T$.
\end{defn}
The notion of 
switching equivalence may also be referred to as
projective unitary equivalence~\cite{chien_characterization_2016}.
Both unitary equivalence and switching equivalence are equivalence relations which---as we 
will see in this section---preserve much of the underlying information of systems of lines.

The following result is related to Proposition 3.12 of~\cite{greaves_frames_2021}, which concerned a sort of scaled unitary equivalence of frames; here, we consider unitary equivalence of any collections of vectors, including frames for their spans.
\begin{lem}
    \label{lem:gramunitiff}
        Let $\Phi =(\phi_j)_{j=1}^n, \Psi = (\psi_j)_{j=1}^n  \se V$  where $V$ is a $d$-dimensional non-isotropic space. 
        $\Phi$ and $\Psi$ are unitarily equivalent if and only if $\Psi^\dagger\Psi=\Phi^\dagger\Phi$
        and $\ker(\Phi)=\ker(\Psi)$. 
\end{lem}

\begin{proof}
        For the forward direction assume that $\Phi$ and $\Psi$ are unitarily equivalent. 
        In which case there exists an isometry $U$ such that $\Psi=U\Phi$ and so 
        $\Psi^\dagger\Psi=\Phi^\dagger U^\dagger U\Phi=\Phi^\dagger\Phi$. 
        And because $U$ is an isometry $\ker(\Phi)=\ker(\Psi)$.

        For the other direction assume that $\Psi^\dagger\Psi=\Phi^\dagger\Phi$ and 
        $\ker(\Phi)=\ker(\Psi)$. We will construct an invertible linear map 
        $A:\im\Phi\rightarrow\im\Psi$ such that $\Psi x=A(\Phi x)$ for any $x\in\F^n$. First 
        we will show that this map is well defined on $\im\Phi$. Let $\Phi x=\Phi y$, meaning 
        $x-y\in \ker(\Phi)$, and notice that because $\ker\Phi=\ker\Psi$ we have that 
        $x-y\in \ker(\Psi)$ and so $A\Phi x=A\Phi y$. So $A$ is a well-defined linear map on the 
        images. $A^{-1}:\image \Psi\ra \image\Phi$ can be constructed in a similar fashion 
        $A^{-1}(\Psi x)=\Phi x$.
        Notice also that because $\Psi^\dagger\Psi=\Phi^\dagger\Phi$, the isomorphism $A$ 
        satisfies $\ip{A\Phi x}{A\Phi y}=\ip{\Psi x}{\Psi y}=\ip{\Phi x}{\Phi y}$ and so 
        $\image\Phi$ and $\image\Psi$ are isometrically equivalent. By Witt's Extension 
        theorem (see, e.g., \cite{jacobson_bilinear_1953}), we know that $A$ extends to an unitary 
        transformation $U:V\ra V$ such that $\Psi=U\Phi$.
\end{proof}

In the case where $\Phi$ and $\Psi$ are frames then $\Psi^\dagger\Psi=\Phi^\dagger\Phi$ implies $\ker(\Phi)=\ker(\Psi)$
in which case $U:V\rightarrow V$ is uniquely determined on all of $V$.

Finally, we note that from \cref{lem:nondegiff} the condition of
$\rk(\Phi)=\rk(\Psi)=\rk(\Phi^\dagger\Phi)$ and 
$\Psi^\dagger\Psi=\Phi^\dagger\Phi$ implies unitary equivalence.

All of the above results for unitary equivalence may be reformulated for switching equivalence.
\begin{lem} \label{lem:gramswitch}
    Let $\Phi =( \phi_j)_{j=1}^n, \Psi = ( \psi_j)_{j=1}^n  \se V$ where $V$ is a $d$-dimensional non-isotropic space. 
    $\Phi$ and $\Psi$ are switching equivalent if and only if 
    $\Psi^\dagger\Psi=T^\dagger\Phi^\dagger\Phi T$ where $T=\diag(t_1,\dots, t_n)$ 
    is a diagonal matrix with $t_it_i^\sigma = 1$ and $\ker(\Phi T)=\ker(\Psi)$.
\end{lem}
\begin{proof}
    For the forward direction assume that $\Phi$ and $\Psi$ are switching equivalent. 
    In that case there exists an isometry $U$ and a diagonal matrix 
    $T=\diag(t_1,\dots, t_n)$ with entries satisfying $t_it_i^\sigma=1$ such that 
    $\Psi=U\Phi T$ meaning $\Psi$ and $\Phi T$ are unitarily equivalent so
    $\Psi^\dagger\Psi=T^\dagger\Phi^\dagger U^\dagger U\Phi T=T^\dagger\Phi^\dagger\Phi T$ and 
    $\ker(\Phi T)=\ker(\Psi)$.
    For the backwards direction assume $\Psi^\dagger\Psi=T^\dagger\Phi^\dagger\Phi T$ 
    and $\ker(\Phi T)=\ker(\Psi)$. Then $\Psi^\dagger\Psi=(\Phi T)^\dagger\Phi T$, so by 
    \cref{lem:gramunitiff}, there exists a isometry such that $\Psi= U\Phi T$.
\end{proof}

\begin{cor}
\label{lem:gramunit-switch}
    Let $\Phi =( \phi_j)_{j=1}^n, \Psi = ( \psi_j)_{j=1}^n  \se V$  where $V$ is a $d$-dimensional non-isotropic space. 
    $\Phi$ and $\Psi$ are unitarily equivalent if $\Psi^\dagger\Psi=\Phi^\dagger\Phi$
    and $\rk(\Phi)=\rk(\Psi)=\rk(\Phi^\dagger\Phi)$.
    
    Likewise $\Phi$ and $\Psi$ are switching equivalent if 
    $\Psi^\dagger\Psi=T^\dagger\Phi^\dagger\Phi T$ where $T=\diag(t_1,\dots,t_n)$ 
    with $t_it_i^\sigma=1$ and 
    $\rk(\Phi)=\rk(\Psi)=\rk(\Phi^\dagger\Phi)$
            
\end{cor}
\begin{proof}
    The second statement implies the first when $T=I$; so, we will only show the second.
    Assume that $\Psi^\dagger\Psi=T^\dagger\Phi^\dagger\Phi T$ and 
    $\rk(\Phi)=\rk(\Psi)=\rk(\Phi^\dagger\Phi)$. Multiplication by $T$ does not
    change the rank so this implies that $\ker(\Phi T)=\ker(T^\dagger\Phi^\dagger\Phi T)=\ker(\Psi)$ 
    because of the original assumptions and that $\ker(\Phi T)\subseteq\ker(T^\dagger\Phi^\dagger\Phi T)$
    and $\ker(\Psi)\subseteq\ker(\Psi^\dagger\Psi)$.
    And so the corollary follows from \cref{lem:gramunitiff}.
\end{proof}

\cref{ex:ranknotiff} shows the necessity of equal ranks for $\Phi$ and 
$\Psi$ in \cref{lem:gramunit-switch}.

\begin{ex}\label{ex:ranknotiff}
    Consider the following matrices which represent collections of vectors in $\F_3^4$ with the standard dot product
    \[\Phi=\begin{bmatrix} 0 & 1\\0 & 1\\0 & 1\\1 & 0\\ \end{bmatrix}\text{ and }
        \Psi=\begin{bmatrix} 0 & 0\\0 & 0\\0 & 0\\1 & 0\\ \end{bmatrix}\]
    Notice that $\Phi$ and $\Psi$ have differing ranks and so are not unitarily
    equivalent. However both have the same Gram matrix
    \[G=\begin{bmatrix} 1 & 0\\0 & 0 \end{bmatrix}\]
    Furthermore $\Phi$ is unitarily equivalent to itself but 
    $\rk(\Phi)\neq \rk(\Phi^\dagger\Phi)$.
\end{ex}

In the complex case switching equivalence can be determined by the $m$-products of 
a system of lines, and in special cases by the triple products. We would like to 
generalize those results here. 
\begin{defn}
    Let $\Phi =( \phi_j)_{j=1}^n \se V$ where $V$ is a $d$-dimensional non-isotropic space. An \textbf{$m$-product} or 
    \textbf{$m$-vertex Bargmann invariant} is defined as
    \[
        \Delta(\phi_{j_1}, \dots, \phi_{j_m}) = \ip{\phi_{j_1}}{\phi_{j_2}}\ip{\phi_{j_2}}{\phi_{j_3}} \dots \ip{\phi_{j_m}}{\phi_{j_1}}.
    \]
    When $m=2$, these are called \textbf{double products}, and when $m=3$, these are
    called \textbf{triple products}.
\end{defn}
Many authors such as \cite{gallagher_orthogonal_1977,brehm_congruence_1998,appleby2011lie,chien_characterization_2016,FICKUS201898,waldron2020tight} have independently studied the structural importance of $m$-products for characterizing line packings in $\bR$, $\bC$, and $\bH$. In this section we wish to generalize some of these results to finite fields.

\begin{lem}\label{lem:mprodsw}
    Let $\Phi =( \phi_j)_{j=1}^n, \Psi = ( \psi_j)_{j=1}^n  \se V$  where $V$ is a $d$-dimensional non-isotropic space.
    If $\Phi$ and $\Psi$ are switching equivalent, then their $m$-products are all equal.
\end{lem}
\begin{proof}
    Since $\Phi$ and $\Psi$ are switching equivalent, there exists a unitary $U:V\ra V$ and 
    $T=\diag(t_1, \dots, t_n)$ with $t_it_i^\sigma=1$ such 
    that  $\psi_j = t_j U \phi_j $ for all $j \in [n]$ 
    Thus for any  $\{j_1, \dots, j_m\} \subseteq [n]$
    \begin{align*}
        \Delta(\psi_{j_1}, \dots, \psi_{j_m}) & = \ip{\psi_{j_1}}{\psi_{j_2}}\ip{\psi_{j_2}}{\psi_{j_3}} \dots \ip{\psi_{j_m}}{\psi_{j_1}}                                                                              \\
                                               & = \ip{t_1 U \phi_{j_1}}{t_2 U\phi_{j_2}}\ip{t_2 U\phi_{j_2}}{t_3 U\phi_{j_3}} \dots \ip{t_m U\phi_{j_m}}{t_1 U\phi_{j_1}}                             \\
                                               & = t_1^\sigma t_2 t_2^\sigma t_3 \dots t_m^\sigma t_1 \ip{U \phi_{j_1}}{U\phi_{j_2}}\ip{U\phi_{j_2}}{U\phi_{j_3}} \dots \ip{U\phi_{j_m}}{U\phi_{j_1}} \\
                                               & = \ip{U \phi_{j_1}}{U\phi_{j_2}}\ip{U\phi_{j_2}}{U\phi_{j_3}} \dots \ip{U\phi_{j_m}}{U\phi_{j_1}}                                                     \\
                                               & =\ip{\phi_{j_1}}{\phi_{j_2}}\ip{\phi_{j_2}}{\phi_{j_3}} \dots \ip{\phi_{j_m}}{\phi_{j_1}}                                                             \\
                                               & =\Delta(\phi_{j_1}, \dots, \phi_{j_m}).
    \end{align*}
\end{proof}

For the remainder of this section we will assume any field $\F=\F_q$ is a finite 
field of case U or O where $q$ is odd. For case U, we have $\F=\F_{q^2}$ with the field involution
$a^\sigma=a^q$, and in Case O, $\F=\F_{q}$ with field involution $a^\sigma=a$. Because $q$ is odd, in both 
cases for any element $aa^\sigma\in \F^\times$, we have that there exists some $\gamma$ such that $\gamma^2=aa^\sigma$. This follows immediately when $\sigma$ is trivial, and for case U, notice that we may choose $\gamma=a^{(q+1)/2}$. With this observation we can define a non-unique square root function $\sqrt{-}:\F_q^{\times 2}\ra \F_q$ such that $\left(\sqrt{a}\right)^2=a$, and we will assume the choice of such a function is fixed throughout, making the following definition well-defined in Case O or U, with odd characteristic.

\begin{defn}\label{defn:phases}
    Let $\Phi =( \phi_j)_{j=1}^n \se V$ where $V$ is a $d$-dimensional non-isotropic space.
    If $\Delta(\phi_j,\phi_k) \neq 0$, define the 
    \textbf{exponential gauge} of $\phi_j$ and $\phi_k$, denoted $\eta_{j,k}$, as
    \[
        \eta_{j,k} = \frac{\ip{\phi_j}{\phi_k}}{\sqrt{\Delta(\phi_j,\phi_k)}}.
    \]
    Otherwise, set $\eta_{j,k} = 0$.
\end{defn}
Note that over $\C$, $\sqrt{\Delta(\phi_j,\phi_k)}$ would be the modulus of the inner product of 
the vectors and thus $\eta_{j,k}$ would be the signum of the inner product, which is the complex 
exponent of the gauge, except when the inner product is zero~\cite{appleby2011lie}.
In case O we note that $\eta_{j,k}=1$ for all $j,k$, and in general $\eta_{j,k}\eta_{k,j}=1$

The following generalizes Lemma 2.1 of \cite{chien_characterization_2016}.
\begin{prop}\label{prop:gaugeeq}
    Let $\Phi =( \phi_j)_{j=1}^n, \Psi = ( \psi_j)_{j=1}^n  \se V$ be frames
    for $V$, a $d$-dimensional non-isotropic space. 
    For each $j \neq k$, let $\eta_{j,k,\phi}$ be the exponential gauges of $\Phi$ and 
    $\eta_{j,k,\psi}$ the exponential gauges of $\Psi$.  Then $\Phi$ and $\Psi$ are 
    switching equivalent if and only if the following two conditions hold:
    \begin{itemize}
        \setlength\itemsep{0em}
        \item[(i)] For all $j, k \in [n]$, $\Delta(\phi_j,\phi_k) = \Delta(\psi_j,\psi_k)$.
        \item[(ii)] For all $j, k \in [n]$, there exist $\beta_j, \beta_k \in\F$  where 
        $\beta_j\beta_j^\sigma=1$ and $\beta_k\beta_k^\sigma=1$
        such that 
        \beq\label{eq:gauge} \eta_{j,k,\psi} = \eta_{j,k,\phi} \beta_j^\sigma \beta_k. \eeq
    \end{itemize}
\end{prop}
\begin{proof} 
    Given $j, k \in [n]$, let $\gamma_{j,k,\phi}:=\sqrt{\Delta(\phi_j,\phi_k)}$ and 
    $\gamma_{j,k,\psi}:=\sqrt{\Delta(\psi_j,\psi_k)}$.

    We first assume that $\Phi$ and $\Psi$ are switching equivalent, i.e., there
    exists a unitary $U:V\ra V$
    and elements $t_1, \hdots, t_n$ where $t_it_i^\sigma=1$ such that $\psi_j =
        t_ j U \phi_j $ for all $j \in [n]$. Let $j, k \in [n]$ with $j \neq k$.
    Thus, \cref{lem:mprodsw} implies that $\Delta(\phi_j,\phi_k) =
        \Delta(\psi_j,\psi_k)$ and further that
    $\gamma_{j,k,\phi}=\gamma_{j,k,\psi}$. Also,
    \begin{align*}
        \eta_{j,k,\psi} \gamma_{j,k,\psi} & = \ip{\psi_j}{\psi_k}  =  \ip{t_jU\phi_j}{t_k U\phi_k} = t_j^\sigma t_k   \ip{U\phi_j}{U\phi_k} \\
                                          & =  t_j^\sigma t_k \ip{\phi_j}{\phi_k} = t_j^\sigma t_k \eta_{j,k,\phi} \gamma_{j,k,\phi}.
    \end{align*}
    If $\gamma_{j,k,\phi}=\gamma_{j,k,\psi}=0$, then $\eta_{j,k,\phi}=\eta_{j,k,\psi}=0$; otherwise, one can divide both sides by $\gamma_{j,k,\phi}$.  In either case, after setting $\beta_j = t_j$ for all $j \in [n]$, condition (ii) holds.\\

    For the other implication, assume that (i) and (ii) hold. (i) implies that for any
    $j \neq k$, $\gamma_{j,k,\phi}=\gamma_{j,k,\psi}$. Let $\tilde{\phi}_j =
        \beta_j \phi_j$ for all $j \in [n]$. Further, we use (ii) to calculate
    \begin{align*}
        \ip{\tilde{\phi}_j}{\tilde{\phi}_k} & = \ip{\beta_j\phi_j}{\beta_k\phi_k} = \beta_j^\sigma \beta_k \ip{\phi_j}{\phi_k} \\
                                                  & = \beta_j^\sigma \beta_k \eta_{j,k,\phi} \gamma_{j,k,\phi}                             \\
                                                  & = \eta_{j,k,\psi} \gamma_{j,k,\psi} = \ip{\psi_j}{\psi_k}.
    \end{align*}
    This shows that $\Psi^\dagger\Psi=T^\dagger\Phi^\dagger\Phi T$ where $T=\diag(\beta_1,\dots,\beta_n)$

    Thus from $\Phi$ and $\Psi$ being frames and \cref{lem:gramunit-switch}, 
    $(\tilde{\phi}_j)_{j=1}^n$ is unitarily equivalent to $\Psi$, implying that $\Phi$ 
    and $\Psi$ are switching equivalent.
\end{proof}

The following generalizes Theorem 2.2 of \cite{chien_characterization_2016}.
\begin{thm}\label{thm:TPs}
    Let $\Phi =( \phi_j)_{j=1}^n, \Psi = ( \psi_j)_{j=1}^n  \se V$ be frames for $V$, where $V$ is a $d$-dimensional non-isotropic space. 
    If $\ip{\phi_j}{\phi_k} \neq 0$ and $\ip{\psi_j}{\psi_k} \neq 0$ for 
    any $j,k \in [n]$ with $j\neq k$, then $\Phi$ and $\Psi$ are switching equivalent 
    if and only if the following two conditions hold:
    \begin{itemize}
        \item[(i)] For all $j, k \in [n]$, $\Delta(\phi_j,\phi_k) = \Delta(\psi_j,\psi_k)$.
        \item[(ii)] For all $j, k, \ell \in [n]$, $\Delta(\phi_j,\phi_k,\phi_\ell) =
                  \Delta(\psi_j,\psi_k,\psi_\ell)$.
    \end{itemize}
\end{thm}
\begin{proof}
    One implication follows immediately from \cref{lem:mprodsw}. For the other 
    implication, suppose that the double and triple products of $\Phi$ and $\Psi$ are 
    equal. For each $j \neq k$, let $\eta_{j,k,\phi}$ be the exponential gauges of 
    $\Phi$ and $\eta_{j,k,\psi}$ the exponential gauges of $\Psi$.  Since the double 
    products are equal, for any $j \neq k$, there exists $\beta_{j,k} \in \F_q$ such 
    that $\Delta(\phi_j,\phi_k) = \Delta(\psi_j,\psi_k) = \beta_{j,k}\neq 0$.  
    Let $\gamma_{j,k} =\sqrt{\beta_{j,k}}\neq 0$. Consider $j \neq k \neq \ell \neq j$. 
    Then
    \begin{align*}
        \eta_{j,k,\phi}\eta_{k,\ell,\phi}\eta_{\ell,j,\phi}\gamma_{j,k}\gamma_{k,\ell}\gamma_{\ell,j} & = \Delta(\phi_j,\phi_k,\phi_\ell) = \Delta(\psi_j,\psi_k,\psi_\ell)=\eta_{j,k,\psi}\eta_{k,\ell,\psi}\eta_{\ell,j,\psi}\gamma_{j,k}\gamma_{k,\ell}\gamma_{\ell,j},
    \end{align*}
    implying that
    \[
        \eta_{j,k,\phi}\eta_{k,\ell,\phi}\eta_{\ell,j,\phi} = \eta_{j,k,\psi}\eta_{k,\ell,\psi}\eta_{\ell,j,\psi}.
    \]
    Fix $\ell$ and note that $\eta_{j,k}^{-1} =\eta_{j,k}^\sigma = \eta_{k,j}$ to
    obtain
    \[
        \eta_{j,k,\psi}=\eta_{j,k,\phi}\left(\eta_{k,\ell,\phi}\eta_{\ell,k,\psi}\right)\left(\eta_{j,\ell,\phi} \eta_{\ell,j,\psi}\right)^\sigma.
    \]
    For any $j \in [n]$, set $\beta_j = \eta_{j,\ell,\phi} \eta_{\ell,j,\psi}$.
    Then (ii) in \cref{prop:gaugeeq} is satisfied, and $\Phi$ and $\Psi$
    are switching equivalent.
\end{proof}

\begin{cor}\label{cor:TP_ETF}
    Let $\Phi =( \phi_j)_{j=1}^n, \Psi = ( \psi_j)_{j=1}^n  \se V$ be 
    $(a,b)$-equiangular systems that are frames for $V$, where $V$ is a $d$-dimensional non-isotropic space.  Then $\Phi$ and $\Psi$ are switching equivalent if and 
    only if for all $j < k < \ell$, $\Delta(\phi_j,\phi_k,\phi_\ell) = \Delta(\psi_j,\psi_k,\psi_\ell)$.
\end{cor}
\begin{proof}
    If $b = 0$, then the Gram matrices of both $\Phi$ and $\Psi$ are $aI_n$, and 
    \cref{lem:gramunit-switch} yields that $\Phi$ and $\Psi$ are unitarily equivalent and 
    thus switching equivalent (with necessarily the same triple products, 
    \cref{lem:mprodsw}).  We now assume $b \neq 0$ and are able to apply 
    \cref{thm:TPs}.  Note that
    \[
        \Delta(\phi_j,\phi_k) = \Delta(\psi_j,\psi_k) = 
        \Big\{ \begin{array}{lr} a^2 & j=k \\ b & j \neq k\end{array}
    \]
    Further note that
    \begin{align*}
        \Delta(\phi_j,\phi_j,\phi_j)    & = a^3                                                                                                 \\
        \Delta(\phi_j,\phi_k,\phi_k)    & = \ip{\phi_j}{ \phi_k}  \ip{\phi_k}{\phi_k}  \ip{\phi_k}{\phi_j} =ab, \quad j \neq k \\
        \Delta(\phi_j,\phi_k,\phi_\ell) & =  \ip{\phi_j}{\phi_k}  \ip{\phi_k}{\phi_\ell}  \ip{\phi_\ell}{\phi_j}
        =  \ip{\phi_\ell}{\phi_j} \ip{\phi_j}{\phi_k}  \ip{\phi_k}{\phi_\ell} = \Delta(\phi_\ell,\phi_j,\phi_k)               \\
        \Delta(\phi_j,\phi_k,\phi_\ell) & =  \ip{\phi_j}{\phi_k}  \ip{\phi_k}{\phi_\ell}  \ip{\phi_\ell}{\phi_j}
        = \left(\ip{\phi_j}{\phi_\ell}  \ip{\phi_\ell}{\phi_k}  \ip{\phi_k}{\phi_j}\right)^\sigma= \left(\Delta(\phi_j,\phi_\ell,\phi_k)\right)^\sigma.
    \end{align*}
    Thus, it suffices to check the triple products for $j < k < \ell$.
\end{proof}

Now we will show that in general, without the requirement of
non-zero scalar products or $(a,b)$-equiangularity, switching equivalence cannot 
be characterized for frames by
the triple products alone and may require the $n$-products where $n$ is the
number of vectors. That is, we will show that there are two frames that have the 
same double and triple products but are not switching equivalent. The following
generalizes Example 2.5 from \cite{chien_characterization_2016}.
\begin{ex}\label{ex:2.5}
    Let $(e_j)_{j=1}^n$ be the standard basis for $\F^n_q$, along with the standard dot product, where $q$ is an odd prime and $n>3$, 
    which is a non-isotropic space.
    For $z\in\{\pm 1\}$ consider the collection of $n+1$ vectors
    
    \[v_j=\begin{cases}
            e_j+e_{j+1} & 1\leq j< n \\
            e_n+ze_{1}  & j=n\\
            e_n & j=n+1
        \end{cases}\]
    where for either choice of $z$, $(v_j)_{j=1}^{n+1}$ forms a frame for $\F_{q}^{n}$ with the 
    following scalar products
    \begin{itemize}  
        \setlength\itemsep{0em}                    
        \item $\ip{v_j}{v_{j+1}}= 1$ for $j \leq n$
        \item $\ip{v_1}{v_{n}}= z$, $\ip{v_{n-1}}{v_{n+1}}= 1$
        \item $\ip{v_i}{v_{j}}= 0$ in all other cases where $i\neq j$
    \end{itemize}
    This means that the only non zero $m$-products of distinct vectors are
    \begin{itemize}
        \setlength\itemsep{0em}
        \item $\Delta(v_j)=2$ for $j\leq n$ and $\Delta(v_{n+1})=1$
        \item $\Delta(v_j, v_{j+1})= 1$ for $j \leq n$, $\Delta(v_{n-1}, v_{n+1})= 1$, and $\Delta(v_1, v_{n})= 1$
        \item $\Delta(v_{n-1},v_{n}, v_{n+1})=1$
        \item $\Delta(v_1,v_2,\dots, v_n)= z$
    \end{itemize}
    This is not an equiangular system as the non-consecutive vectors have scalar
    products that are zero. 
    For either choice of $z$ the $m$-products are equivalent for all $m<n$ and differ only 
    for the $n$-products. Since the $n$-products are not equal, \cref{lem:mprodsw} implies the two sequences are not switching equivalent.
\end{ex}
In general, switching equivalence of collections of vectors (regardless of if they form frames) is completely determined by the $m$-products for all $m\leq n$.  This result over $\bC$ is due to  \cite{chien_characterization_2016}, but their proof holds for any non-isotropic spaces with almost no modifications, which we emulate here. We note that this result was discovered in an equivalent form years earlier in \cite{gallagher_orthogonal_1977}.
\begin{prop}
Let $\Phi =( \phi_j)_{j=1}^n, \Psi = ( \psi_j)_{j=1}^n  \se V$ be collections of
vectors in $V$, a $d$-dimensional non-isotropic space with $\ker(\Phi)=\ker(\Psi)$.
Then $\Phi$ and $\Psi$ are switching equivalent if and only if 
$\Delta(\phi_{j_1},\phi_{j_2},\dots,\phi_{j_m})=\Delta(\psi_{j_1},\psi_{j_2},\dots,\psi_{j_m})$ 
for all $m\leq n$ where 
$1\leq j_{\ell}\leq n$ for all $\ell$, that is all $m$-products for $\Phi$ and $\Psi$ are equal.
\end{prop}
\begin{proof}
    We have already shown the forward direction in \cref{lem:mprodsw}. 
    For the other direction
    we follow the same approach used in \cite{chien_characterization_2016} which utilizes the correlation network of a frame, which was introduced in \cite{Strawn2009GeometryAC}.

    Assume that the $m$-products of $\Phi$ and $\Psi$ are equal.
    Let $\Gamma(\Phi)$ be the correlation network of $\Phi$: a graph whose vertices 
    are the indices $1,2,\dots,n$ where vertex $j$ corresponds to 
    the vector $\phi_j$, and an edge, denoted $(j,k)$, exists between $j$ and $k$ when $\ip{\phi_j}{\phi_k}\neq 0$.
    Notice that $\Delta(\phi_j,\phi_k)=0$ if and only if $\ip{\phi_j}{\phi_k}=0$ meaning $\Gamma(\Phi)=\Gamma(\Psi)$.
    Call their common correlation network $\Gamma$. We may assume that $\Gamma$ is 
    connected, as if it were not we could restrict ourselves to looking at each component individually. 
    Because $\Gamma$ is connected there exists a
    spanning tree $T$ for $\Gamma$. Fix a root vertex $r$ and let $c_r=1$. Let $v$ be a child vertex of $r$ 
    and because $\Delta(\phi_r,\phi_v)=\Delta(\psi_r,\psi_v)$ we know that there exists some unimodular constant
    $c_v$(i.e. $c_v^\sigma c_v=1$) where $\ip{\phi_r}{\phi_v}=c_v^\sigma\ip{\psi_r}{\psi_v}=\ip{c_r\psi_r}{c_v\psi_v}$. Repeating this 
    process inductively we have that for every non-root vertex $k$ whose parent vertex 
    is $j$ in the tree, that 
    \[c_k=\left(\frac{\ip{\phi_j}{\phi_k}}{\ip{c_j\psi_j}{\psi_k}}\right)^\sigma\]
    which satisfies $\Delta(\phi_j,\phi_k)=\Delta(\psi_j,\psi_k)=\Delta(c_j\psi_j,c_k\psi_k)$
    and $\ip{\phi_j}{\phi_k}=c_k^\sigma\ip{c_j\psi_j}{\psi_k}=\ip{c_j\psi_j}{c_k\psi_k}$.
    This process gives us a collection of unimodular constants $c_1,\dots, c_n$ for each vertex of $\Gamma$ such that by construction
    $\ip{\phi_j}{\phi_k}=\ip{c_j\psi_j}{c_k\psi_k}$ whenever $(j,k)$ is an edge in the spanning tree $T$.

    Now, we must show that our chosen constants respect the other non-zero scalar products.
    Let $e=(\ell,k)$ be an edge of the graph $\Gamma$ such that $e$ is not an edge of $T$. The addition of $e$ into $T$ creates
    a unique cycle in $e\cup T$ with vertices $(j_1,j_2,\dots,j_{m-2},\ell,k)$ of $m\leq n$ vertices.
    By the initial assumption and the fact that each $c_j$ is unimodular we know that 
    \[\Delta(\phi_{j_1},\dots,\phi_{j_{m-2}},\phi_{\ell},\phi_{k})=\Delta(\psi_{j_1},\dots,\psi_{j_{m-2}},\psi_{\ell},\psi_{k})=\Delta(c_{j_1}\psi_{j_1},\dots,c_{j_{m-2}}\psi_{j_{m-2}},c_\ell\psi_{\ell},c_k\psi_{k})\]
    This means that 
    \begin{multline*}
        \ip{\phi_{j_1}}{\phi_{j_2}}\cdots\ip{\phi_{j_{m-2}}}{\phi_{\ell}}\ip{\phi_{\ell}}{\phi_{k}}\ip{\phi_{k}}{\phi_{j_1}}\\
        =\ip{c_{j_1}\psi_{j_1}}{c_{j_2}\psi_{j_2}}\cdots\ip{c_{j_{m-2}}\psi_{j_{m-2}}}{c_\ell\psi_{\ell}}\ip{c_{\ell}\psi_{\ell}}{c_k\psi_{k}}\ip{c_k\psi_{k}}{c_{j_1}\psi_{j_1}}
    \end{multline*}
     From the construction of the unimodular constants we know that 
     $\ip{\phi_i}{\phi_j}=\ip{c_i\psi_i}{c_j\psi_j}$ when $(i,j)\in T$ so it must be the 
     case that 
     $\ip{\phi_\ell}{\phi_k}=\ip{c_\ell\psi_\ell}{c_k\psi_\ell}$ for the added edge $e$.
\end{proof}

\section{Combinatorial Designs: Two-Graphs and Quasi-Symmetric Designs}
\label{sec:combinatorics-two-graphs}
In this section, we will present a short overview on $t$-designs, two-graphs, and quasi-symmetric designs which are used frequently in this paper along with some examples connecting them to systems of lines in orthogonal geometries. For a more complete overview see \cite{Cameron_Lint_1991,taylortwographs}.
In the most basic sense, combinatorial design theory concerns collections of subsets of a given set which have certain additional properties.  To that end, the following notation will be useful. If $\Omega$ is a finite set with $n$ elements, then for $k \leq n$, $\Omega^{\{k\}}$ is the collection of all $k$-element subsets of $\Omega$.
A \textbf{$t$-$(n,k,\lambda)$ design} is a pair $(\Omega,\mathcal B)$ where $\Omega$ is a set of $n$ points and $\mathcal B \subseteq\Omega^{\{k\}}$, where the elements of $\cB$ are called \textbf{blocks}, are such that any $t$ points are contained in exactly $\lambda$ blocks. If $r$ is the number of blocks in a $t$-design, with $t \geq 2$, that contain any given point, then $|\mathcal B|k=nr$ and $r(k-1)=(n-1)\lambda$. 

\begin{lem}[Fisher's inequality]\label{lem:fish}
    In a $2$-design if $k<n$ then $|\mathcal B|\geq n$. And in the case where $|\mathcal B|=n$ then $r=k$
\end{lem}
A $t$-design is called a \textbf{quasi-symmetric $t$-$(n,k,\lambda; s_1, s_2)$ design} if the intersection of any two distinct blocks has $s_1$ or $s_2$ points, such that $s_1\leq s_2$. It can be shown that $s_2-s_1$ divides $k-s_1$ and $r-\lambda$. As an example, one could consider the lines in a plane over a finite field $\bF_q^2$. Any pair of points defines a unique line, meaning that the lines define a $2$-$(q^2,q,1)$ design.  Furthermore, each pair of lines either intersects in a unique point or  are parallel.  Thus, the lines form a quasi-symmetric $2$-$(q^2,q,1; 0, 1)$ design.

Another type of design, which is intimately related to equiangular lines in $\bR^d$, and which we will show is also related to equiangular lines in orthogonal geometries, is a two-graph.

\begin{defn}
    Let $\Omega$ be a set of size $n$, called the \textbf{point set} and $\mathcal B\se \Omega^{\{3\}}$, called the \textbf{coherent triples}. A pair $(\Omega,\mathcal B)$ is a \textbf{two-graph} if every $4$-element subset of $\Omega$ contains an even number of elements of $\mathcal B$ as subsets.

    A two-graph $(\Omega,\mathcal B)$ is called \textbf{regular} if every $2$ element subset of $\Omega$ is contained in an equal number of the coherent triples in $\mathcal B$. In this case $(\Omega,\mathcal B)$ is a $2$-$(n,3,\ell)$ design for some integer $\ell$.
\end{defn}

Let $(\Omega,\mathcal B)$ be a two-graph. Then a subset $\Gamma\se \Omega$ is \textbf{coherent} if $|\Gamma|\geq 3$ and every $3$-element subset of $\Gamma$ is a coherent triple, that is $\Gamma^{\{3\}}\se\mathcal B$. Likewise $\Gamma\se \Omega$ is \textbf{incoherent} if $|\Gamma|\leq 2$ or no $3$-element subset of $\Gamma$ is a coherent triple, that is $\Gamma^{\{3\}}\cap\mathcal B=0$. If $(\Omega,\mathcal B)$ is a regular two-graph, and therefore a $2$-$(n,3,\ell)$ design, then every coherent triple is contained in $m$ coherent quadruples where $|\Omega|=3\ell-2m$ and in which case we will refer to $(\Omega,\mathcal B)$ as a \textbf{regular two-graph with parameters $(n,\ell,m)$}.

Given an undirected graph $G=(V,E)$ where $n=|V|$ we can construct a two-graph $(V,\mathcal B_G)$ where $\mathcal B_G$ is the set of all triples of vertices where the number of edges in the induced subgraph is odd.
Two graphs $G_1$ and $G_2$, with the same vertex set $V$ are called \textbf{switching equivalent} if there is a set of vertices $X\se V$ where switching all edges and non-edges of $G_1$ between $X$ and the compliment $G_1-X$ result in $G_2$. Graphs that are switching equivalent give rise to equal two-graphs. In fact any two-graph $(\Omega,\mathcal B)$, represents a class of switching equivalent graphs since in any induced subgraph on $3$ vertices, switching changes and even number of edges.

We will make use of the \textbf{Seidel adjacency matrix} $S\in\Z^{n\times n}$ to connect two-graphs to equiangular lines. Label the vertices in $V$ as $[n]$, and then define $S$ as follows
\[S_{ij}:=\begin{cases}
    -1 & \text{If }i\text{ and }j\text{ are adjacent}\\
    0 & \text{If }i=j\\
    1 & \text{If }i\text{ and }j\text{ are not adjacent}
\end{cases}\]
The spectrum of a Seidel adjacency matrix is invariant under switching equivalence, and it is the spectrum which is critical in the connection to equiangular lines. So, we can sensibly define a non-unique Seidel adjacency matrix for any two-graph to be the Seidel adjacency matrix of any graph that induces the two-graph. An equivalent condition for a regular two-graph is that the spectrum of any of its Seidel adjacency matrices has exactly two eigenvalues.
Furthermore let $G_x$ be the unique graph which is switching equivalent to $G$ and where $x$ is isolated, and removed. Then a non-trivial, non-complete two-graph is regular if and only if there exists some $x\in V$ where $G_x$ is strongly regular with parameters $k=2\mu$. In this case the parameters of the regular two-graph and the strongly regular graph align, in that $k=\ell$ and $\lambda=m$.

\begin{defn}
    A graph $G=(V,E)$ is a \textbf{strongly regular graph (SRG)} with parameters $(v,k,\lambda,\mu)$ if it is not complete nor empty and has $n=|V|$, $k$ the degree or valency of each vertex, $\lambda$ the number of shared neighbors for any pair of adjacency vertices, and $\mu$ the number of shared neighbors for any pair of non-adjacency vertices.

    A graph $G=(V,E)$ is a \textbf{$p$-modular strongly regular graph (SRG$_p$)} with parameters $(v,k,\lambda,\mu)$ if it is not complete nor empty and has $n=|V|$, $k$ is equivalent to the degree of each vertex modulo $p$, $\lambda$ is equivalent to the number of shared neighbors for any pair of adjacency vertices modulo $p$, and $\mu$ is equivalent to the number of shared neighbors for any pair of non-adjacency vertices modulo $p$.
\end{defn}
We use SRGs and SRG$_p$s only in the proof of \cref{lem:etfiffregulartwograph}, and so we will not give an in-depth overview. For an overview of SRGs and their connections two-graphs we point to \cite{Cameron_Lint_1991} and for an overview of SRG$_p$s we point to \cite{greaves_frames_2021-1}.

Let $\Phi$ be an $(a,b)$-equiangular system of $n$ lines in an orthogonal geometry $W$ of characteristic not equal to $2$, where $b\neq 0$. 
In this case the angle between any two vectors is either $\beta$ or $-\beta$ where $\beta^2=b$. This suggests we can consider $\frac{1}{\beta}(\Phi^\dagger\Phi-aI)$ to be the image of a Seidel adjacency matrix for a graph $G=(V,E)$ under the map $\pi:\Z\ra \F$, where the vertex set $V=\Phi$ are the vectors, and there is a edge between $\phi_i$ and $\phi_j$ if $\ip{\phi_i}{\phi_j}=-\beta$. 
This means that $\Phi$ gives rise to a well-defined two-graph $(\Phi,\mathcal C)$ where $\mathcal C$ is the set of all triples of vectors $\{\phi_j,\phi_k,\phi_\ell\}$ where $\Delta(\phi_j,\phi_k,\phi_\ell)=\ip{\phi_j}{\phi_k}\ip{\phi_k}{\phi_\ell}\ip{\phi_\ell}{\phi_j}=-\beta^3$. This conditions exactly corresponds to the triple $\{\phi_j,\phi_k,\phi_\ell\}$ having an odd number of edges in the graph corresponding to $\Phi$. The ``correct'' choice of $\beta$ is not immediately obvious as finite fields are not ordered, and we will explore the question of choosing a particular $\beta$ in \cref{ssec:incoherentreality}. Just as for graphs, systems of equiangular lines which are switching equivalent, give rise to equal two-graphs: \cref{cor:TP_ETF} tells us that the parity of negative inner products in any triple product uniquely determines switching equivalence for ETFs in $\bR^d$.

Likewise in some cases a two-graph can be used to construct a system of equiangular lines over an orthogonal geometry of dimension equal to the multiplicity of any non-zero eigenvalue.
It is however not always the case that the two-graphs constructed from systems of equiangular lines over orthogonal geometries themselves give rise to switching equivalent systems of lines or even lines in the same dimensional space, and so we caution the reader: for the preservation of sensible structure, restrictions on $\beta$ and the characteristics must be made.

\begin{ex}\label{ex:twonotswitch}
    Consider $\F_{11}^2$ in the real model, with the standard dot product.
    \[\Phi=\begin{bmatrix}
        0 & 3 & 8\\
        1 & 5 & 5
    \end{bmatrix}\;\;\;\; \Phi^\dagger\Phi = \begin{bmatrix}
        1 & 5 & 5\\
        5 & 1 & 5\\
        5 & 5 & 1
    \end{bmatrix}\]
    $\Phi$ is a $(1,3,7)$-ETF. By picking $\beta=5$ we can see that the resulting two-graph is the trivial two-graph $(\{1,2,3\},\{\})$, a graph that induces this two-graph is the trivial graph on $3$ vertices, which has Seidel adjacency matrix $J-I\in\Z^{3\times 3}$ with eigenvalues $2$ with multiplicity $1$ and $-1$ with multiplicity $2$.
    This means we can construct $-I-J+I=-J$ which has $-1$ on the diagonal, and $-1$ off the diagonal. The matrix $-J$ is rank $1$ meaning this gives rise to a $(-1,1)$-equiangular frame for a $1$ dimensional orthogonal geometry. If instead we used the eigenvalue of $2$ we would get $2I-J+I=3I-J$
    which would be a rank $2$ matrix and therefore the Gram matrix of a $(2,1)$-equiangular frame for a $2$ dimensional orthogonal geometry, and therefore an ETF. Picking $\beta=-5\equiv 6$ would result in a similar behavior.
\end{ex}

\section{Sufficient Conditions for Tightness} 
\label{sec:beigntight}
In this section we will explore sufficient conditions for collections of lines to be tight.

In \cref{sec:frames} we defined tightness to be a property of frames. However, as we will see in this section, many equivalent notions of tightness are enough to guarantee that a collection of vectors has a non-isotropic image, and hence would be a tight frame for their span.

\subsection{\texorpdfstring{\cref{lem:tfaetight1}}{Lemma 2.4} Revisited}
\begin{lem}
    \label{lem:tfaetight2}
    Let $c\neq 0$ and $\Phi=(\phi_j)_{j=1}^n$ be a collection of vectors in a non-isotropic space $V$, 
    then the following are equivalent and imply that $\Phi:\F^n\ra \image \Phi$ is a 
    $c$-tight frame for $\image \Phi$.
    \begin{itemize}
        \setlength\itemsep{0em}
        \item[(i)] $\Phi\Phi^\dagger=cI$ on $\image\Phi$
        \item[(ii)] $\Phi\Phi^\dagger\Phi=c\Phi$
    \end{itemize}
\end{lem}
\begin{proof}
    First we will show (i) implies (ii) which in turn implies $\Phi:\F^n\ra \image \Phi$ is a $c$-tight frame.
    Notice that (i) immediately implies (ii) and by counting ranks we can conclude from
    $\Phi\Phi^\dagger\Phi=c\Phi$ where $c\neq 0$ that
    \[\rank(c\Phi)=\rank(\Phi\Phi^\dagger\Phi)\leq 
    \min(\rk(\Phi), \rk(\Phi^\dagger\Phi))= \rk(\Phi^\dagger\Phi),\]
    where the last equality follows from $\rk(\Phi)\geq \rk(\Phi^\dagger\Phi)$.
    So $\rk(\Phi)=\rk(\Phi^\dagger\Phi)$. By \cref{lem:nondegiff} we know 
    that $\image\Phi$ is non-isotropic. This shows $\Phi:\F^n\ra \image \Phi$ is a frame. 
    This also implies (i) from (ii) by \cref{lem:tfaetight1}.
\end{proof}
We note that this result does not guarantee that the $\discr(\image \Phi)$ agrees with $\discr(V)$, and in general it will not. An explicit example is shown in \cref{ex:discrisweird}.
\begin{defn}
    Let $M=[M_{ij}]_{i,j\in[n]}$ be a square matrix, and denote the columns by $m_1,\dots,m_n$. 
    Select a subset of the columns $(m_j)_{i\in K}$ that forms a basis for $\image M$. Define a \textbf{basic submatrix} of $M$ to be $M_b=[M_{ij}]_{i,j\in K}$
\end{defn}
In case O, \cite{greaves_frames_2021}, showed that the $\discr(\im\Phi)=(\det((\Phi^\dagger\Phi)_b))\F_q^{\times 2}$ for any basic submatrix $(\Phi^\dagger\Phi)_b$ of the Gram matrix.

\begin{ex}
    \label{ex:discrisweird}
    let $V=\F_5^3$ be an orthogonal geometry in the real model where $\discr(\F_5^3)$ is trivial
    and consider the system of lines
    \[\Phi = \begin{bmatrix}
        0 & 2 & 3\\
        4 & 2 & 2\\
        4 & 2 & 2
    \end{bmatrix}\;\;\;\;\;\;\;\;\Phi^\dagger\Phi = \begin{bmatrix}
        2 & 1 & 1\\
        1 & 2 & 4\\
        1 & 4 & 2
    \end{bmatrix},
    \]
    which is a $(2, 1)$-equiangular system of lines. We can also see that $\dim(\image \Phi)=2$. From \cref{lem:tfaetight2} and because $3\Phi=\Phi\Phi^\dagger\Phi$ we know that $\Phi$ is an $(2, 1,3)$-ETF for $\image\Phi$. Looking at a basic submatrix formed from the first two columns of $\Phi^\dagger\Phi$, whose determinant is $3$, we can see that the discriminant of $\Phi$ is not a square. Now consider the orthogonal geometry $\F_5^2$ with a scalar product having Gram matrix $\diag(1,3)$ and consider the frame for $\F_5^2$,
    \[\Psi= \begin{bmatrix}
        0 & 2 & 3\\
        2 & 1 & 1
    \end{bmatrix}\]
    which is also a $(2, 1,3)$-ETF where $\Phi^\dagger\Phi=\Psi^\dagger\Psi$.
\end{ex}

We note that case (iii) from \cref{lem:tfaetight1} does not imply
that $\image\Phi$ is non-isotropic in general, as $\rk(\Phi)$ and $\rk(\Phi^\dagger\Phi)$ may not 
agree. However we provide a few situations where (iii) does imply $\image\Phi$ is non-isotropic. First we give conditions on square matrices which make them the Gram matrices for some frame.

\begin{prop} (Theorem 3.13 \cite{greaves_frames_2021})
    \label{prop:whenGisETFUnitary}
    Let $G$ be an $n\times n$ matrix with entires in $\F_{q^2}$. Then $G$ is the Gram matrix of some frame for a unitary geometry $\F_{q^2}^d$ if and only if
    \begin{itemize}
    \setlength\itemsep{0em}
     \item[(i)] $G=G^*$
     \item[(ii)] $\rk(G)=d$
    \end{itemize}
    Additionally $G$ is the Gram matrix for a $c$-tight frame with $c\in \F_{q^2}$ if and only if
    \begin{itemize}
        \item[(iii)] $G^2=cG$
    \end{itemize}
    Furthermore $G$ is the Gram matrix of some $(a,b,c)$-ETF for $a,b\in\F_{q^2}$ if and only if
    \begin{itemize}
        \setlength\itemsep{0em}
     \item[(iv)] $G_{ii}=a$ for all $i\in[n]$
     \item[(v)] $G_{ij}G_{ji}=b$ for all $i\neq j$ in $[n]$
    \end{itemize}
 \end{prop}

\begin{prop} (Theorem 3.15 \cite{greaves_frames_2021}, Proposition 2.12 \cite{greaves_frames_2021-1})
    \label{prop:whenGisETF}
    Let $G$ be an $n\times n$ matrix with entires in $\F_q$ for some odd prime power $q$. Then $G$ is the Gram matrix of some frame for an orthogonal geometry $\F_{q}^d$ if and only if
    \begin{itemize}
        \setlength\itemsep{0em}
     \item[(i)] $G=G^\intercal$
     \item[(ii)] $\rk(G)=d$
     \item[(iii)] $(\det (G_b))=\discr(\F_q^d)$
    \end{itemize}
    Additionally $G$ is the Gram matrix for a $c$-tight frame with $c\in \F_{q}$ if and only if
    \begin{itemize}
        \item[(iv)] $G^2=cG$
       \end{itemize}
       Furthermore $G$ is the Gram matrix of some $(a,b,c)$-ETF for $a,b\in\F_q$ if and only if
    \begin{itemize}
        \setlength\itemsep{0em}
     \item[(v)] $G_{ii}=a$ for all $i\in[n]$
     \item[(vi)] $G_{ij}^2=b$ for all $i\neq j$ in $[n]$
    \end{itemize}
\end{prop}
We also give a more direct instance where (iii) from \cref{lem:tfaetight1} implies
that $\image\Phi$ is non-isotropic.
The proof uses $CR$-decompositions which in some cases can fill the gap from the lack of spectral theory, and positive definiteness, in the positive characteristic setting.
\begin{prop}
    Let $M\in \F^{n\times m}$ be any matrix over any field, with $\rk(M)=r$.
    Then there exist matrices $C\in \F^{n\times r}$ and $R\in \F^{r\times m}$ such that $M = CR$ and
    $\rk(C)=\rk(R)=\rk(M)=r$.
\end{prop}
\begin{proof}
    Let $C$ be a matrix with $r$ linearly independent columns of $M$. This means $C\in \F^{n\times r}$ and has rank $r$.
    Then because every column of $M$ is a linear combination of the columns of $C$ we can construct the matrix 
    $R\in \F^{r\times m}$ where the $i$th column of $R$ is the coefficients in the linear combination of 
    the columns of $C$ to make the $i$th column of $M$.
    This gives us $M=CR$. Finally notice that $\rk(R)= r$, because if it were less than the product $CR$ would not have rank $r$.
\end{proof}

This decomposition gives us some very nice properties. It is important to note that this decomposition is not unique.
Because $C:\F^r\ra \F^n$ is full rank it is injective and so has a left inverse $C^{-}$ where $C^{-}C=I_{r}$ and because 
$R:\F^n\ra \F^r$ is full rank it is surjective, and so has a right inverse $R^{-}$ where $RR^{-}=I_r$. 

\begin{thm}
    \label{thm:buildingcharacter}
    Let $\Phi=(\phi_j)_{j=1}^n$ be an equal norm collection of vectors in a non-isotropic space $V$, with $\ip{\phi_j}{\phi_j}=a$ for all $j$, such that $\frac{na}{d}\neq 0$ where $d:=\dim(\image\Phi)$.
    If $\Char\F>d$ and $(\Phi^\dagger\Phi)^2=\frac{na}{d}\Phi^\dagger\Phi$
    Then $\Phi$ is an $\frac{na}{d}$-tight frame for $\image\Phi$.
\end{thm}
\begin{proof}
    Let $\Phi=(\phi_j)_{j=1}^n$ be a collection of equal norm vectors with $\ip{\phi_j}{\phi_j}=a\neq 0$ in a non-isotropic space $V$ such that $d:=\dim(\image\Phi)$, over a field $\F$ with $\Char\F>d$ and  $(\Phi^\dagger\Phi)^2=\frac{na}{d}\Phi^\dagger\Phi$.
    
    Consider a $CR$-decomposition for the Gram matrix $\Phi^\dagger\Phi=CR$ 
    where $C\in \F^{n\times r}$ and $R\in \F^{r\times n}$, 
    such that $r=\rk(\Phi^\dagger\Phi)=\rk(C)=\rk(R)$. This gives us the following 
    \begin{align*}
        CRCR&=\frac{na}{d}CR\\
     \Rar \quad   C^{-}CRCRR^{-}&=\frac{na}{d}C^{-}CRR^{-}\\
   \Rar \quad     RC&=\frac{na}{d}I_r
    \end{align*}
    
    Now looking at the trace we can determine 
    \[na = \tr(\Phi^\dagger\Phi)=\tr(CR)=\tr(RC)=\tr\left(\frac{na}{d}I_r\right)=\frac{na}{d}\rk(\Phi^\dagger\Phi),\]
where the equalities are in $\bF$. Since $na/d \neq 0$ and $\Char\F > d$, we conclude that as integers
    $\rk(\Phi^\dagger\Phi)= d$.  Thus, \cref{lem:nondegiff} implies that $\Phi:\F^n\ra \image\Phi$ is a frame and further by \cref{lem:tfaetight1} an $\frac{na}{d}$-tight frame.
\end{proof}

\subsection{Naimark Complements}
\label{ssec:Naimarking}  

In Hilbert spaces, every tight frame has a (non-unique) Naimark complement, which preserves some of the key properties of the frame~\cite{wandering1998dai,neumark1943representation}.  Over finite dimensions, the Naimark complement of a tight frame (respectively, ETF) of $n$ vectors for a $d$-dimensional space always is a tight frame (respectively, ETF) of $n$ vectors for an $(n-d)$-dimensional space.  This decreases the parameter space which one needs to explore when analyzing such frames.  For more on the properties of the Naimark complement over the reals and complexes and issues when trying to extend the complement to non-tight frames, see~\cite{casazza_every_2013,king2025impossibility}.
In this section, we expand the theory of Naimark complements in the case of 
frames over fields of non-zero characteristic building off of the work in \cite{greaves_frames_2021}.
In the real and complex setting a Naimark complement need only satisfy $cI_n=\Phi^\dagger\Phi + \Psi^\dagger\Psi$ for $c$ the tight frame bound.
From this it can be shown that $\Phi\Psi^\dagger = 0$ and $\Psi\Phi^\dagger = 0$, through the positive definiteness 
of inner products, which in essence means the stacked matrix $\begin{bmatrix} \Phi\\ \Psi\end{bmatrix}$ would be a scaled unitary.
This single condition is not quite sufficient in the more general setting. 
\begin{ex}\label{ex:naiweird}
    Consider the following matrices
    \[\Phi=\begin{bmatrix} 1& 1\end{bmatrix}\quad \textrm{and} \quad \Psi=\begin{bmatrix} 1& 2\\ 0&1\\ 0& 1\\ 0 & 1\end{bmatrix},\]
    where $\Phi$ is a $2$-tight frame in the real model for $\F_3^1$ and $\Psi$ is a system of lines in the real model for $\F_3^4$.
    Notice that 
    \[
    2I=\Phi^\dagger\Phi+\Psi^\dagger\Psi = \begin{bmatrix} 1 & 1 \\ 1 & 1 \end{bmatrix} + \begin{bmatrix} 1 & 2 \\ 2 & 1\end{bmatrix}; 
    \]
    however, $\Psi$ is not even a frame for its image since $\rk(\Psi^\dagger\Psi) =1 < 2 = \rk(\Psi)$ (\cref{lem:nondegiff}).
    We can also see that 
    \[\Psi\Phi^\dagger=\begin{bmatrix}
        0&1&1&1
    \end{bmatrix}^\intercal.\]
\end{ex}

\begin{defn}
    Let $\Phi:\F^n\ra V$
    be a $c$-tight frame ($c\neq 0$) for a $d$-dimensional non-isotropic space $V$.
    Then the synthesis operator $\Psi:\F^n\ra W$ for a collection of vectors
    is a \textbf{Naimark complement} if $cI_n=\Phi^\dagger\Phi + \Psi^\dagger\Psi$ and 
    $\Psi\Phi^\dagger = 0$ 
\end{defn}
We also note that the geometries of $\image\Phi$ and $\image\Psi$ need not be the same, in the sense that they need not share a discriminant.

\begin{lem}\label{lem:naimarksarecomplentary}
    Let $\Phi:\F^n\ra V$
    be a $c$-tight frame ($c\neq 0$) for a $d$-dimensional non-isotropic space $V$ and $\Psi:\F^n\ra W$ 
    a Naimark complement then $\image\Psi$ is non-isotropic and
    $\Psi:\F^n\ra \image\Psi$
    is a $c$-tight frame with $\dim(\image\Psi) = n-d$. 
    Additionally $\image\Psi^\dagger = (\image\Phi^\dagger)^\perp$, and $\discr(\image\Psi)=c^n\discr(V)$.
    Furthermore, if $\Phi$ is a $(a,b,c)$-ETF, then $\Psi:\F^n\ra \image\Psi$ is 
    an $(c-a,b,c)$-ETF.
\end{lem}
\begin{proof}
    Using the fact that $\Psi\Phi^\dagger=0$ and because $\F^n$, $V$, 
    and $W$ are all non-isotropic,
    we have that $(\Psi\Phi^\dagger)^\dagger=\Phi\Psi^\dagger=0$.
    We can also determine that
    \begin{align*}
        \begin{split}
        cI_n&=\Phi^\dagger\Phi + \Psi^\dagger\Psi\\
    \Rar \quad    c\Psi&=\Psi\Phi^\dagger\Phi + \Psi\Psi^\dagger\Psi\\
\Rar \quad        c\Psi&= \Psi\Psi^\dagger\Psi
        \end{split}
    \end{align*}
    
    By \cref{lem:tfaetight2} we know 
    that since $\image\Psi$ is non-isotropic,
    $\Psi:\F^n\ra \image\Psi$ is a $c$-tight frame.
    
    Finally, we will look at the dimension of $\image\Psi$.
    Notice that because $cI_n=\Phi^\dagger\Phi + \Psi^\dagger\Psi$ and the non-isotropy, we have that
    \begin{equation*}
        n=\rk(cI_n)=\rk(\Phi^\dagger\Phi + \Psi^\dagger\Psi)\leq \rk(\Phi^\dagger\Phi) + \rk(\Psi^\dagger\Psi) = d + \rk(\Psi);
    \end{equation*}
    so, $\rk(\Psi)\geq n-d$.
    Because $\Phi\Psi^\dagger=0$ we know that $\image\Psi^\dagger \se \ker\Phi$
    meaning $\dim(\image\Psi^\dagger) \leq n-d$, i.e., $\rk(\Psi)= n-d$.
    
    Notice that $\image(\Psi^\dagger\Psi)=\image(cI-\Phi^\dagger\Phi)=\image(\Phi^\dagger\Phi)^\perp$, 
    and so by \cref{lem:kerwhenframe}, 
    we have that $\image(\Psi^\dagger)=\image(\Phi^\dagger)^\perp$. 
    
    This also means that $\F^n=\image(\Psi^\dagger)\oplus\image(\Phi^\dagger)$. Thus, $\discr(\image(\Psi^\dagger))\discr(\image(\Phi^\dagger))=\discr(\F^n)=\F^{\times 2}$, giving us that $\discr(\image(\Psi^\dagger))=\discr(\image(\Phi^\dagger))$.
    It follows from Lemma 3.18 of \cite{greaves_frames_2021} that $\discr(\image(\Phi^\dagger))=\det(\Phi^\dagger\Phi)\discr(V)$
    and $\discr(\image(\Psi^\dagger))=\det(\Psi^\dagger\Psi)\discr(\image\Psi)$. Putting this together, we get that
    $\discr(\image\Psi)=c^d\discr(V)/c^{n-d}=c^{-n}\discr(V)=c^{n}\discr(V)$.

    Finally, assume that $\Phi$ was an $(a,b,c)$-ETF.
    By $cI_n=\Phi^\dagger\Phi + \Psi^\dagger\Psi$ we know 
    $\ip{\psi_j}{\psi_j}=c-\ip{\phi_j}{\phi_j}$ and for $j\neq k$ that
    $\ip{\psi_j}{\psi_k}=-\ip{\phi_j}{\phi_k}$, and so 
    $\ip{\psi_j}{\psi_k}\ip{\psi_k}{\psi_j}=(-1)^2\ip{\phi_j}{\phi_k}\ip{\phi_k}{\phi_j}=b$
\end{proof}

We note that if $\Psi$ were known to be a frame for $W$ then the condition $\Psi\Phi^\dagger=0$ follows from $\Phi^\dagger\Phi+\Psi^\dagger\Psi=cI$.
Naimark complements always exists for non-degenerate tight frames over finite fields but the discriminant will depend on the field and $c$.

\begin{thm}\label{thm:naimarkexistsU} (Proposition 3.22 \cite{greaves_frames_2021})
    Let $\Phi:\F_{q^2}^n\ra V$ be a $c$-tight frame ($c\neq 0$) for a $d$-dimensional unitary geometry $V$. Then there exists a Naimark complement $\Psi:\F_{q^2}^{n}\ra W$ for an $(n-d)$-dimensional unitary geometry $W$.
\end{thm}
\begin{thm}\label{thm:naimarkexistsO} (Proposition 3.23 \cite{greaves_frames_2021})
Let $\Phi:\F_{q}^n\ra V$ be a $c$-tight frame ($c\neq 0$) for a $d$-dimensional orthogonal geometry $V$. Then there exists a Naimark complement $\Psi:\F_{q}^{n}\ra W$ for an $(n-d)$-dimensional orthogonal geometry $W$ with $\discr(W)=c^{n}\discr(V)$.
\end{thm}
In fact Proposition 3.23 of \cite{greaves_frames_2021} gives a stronger statement allowing $\Psi^\dagger\Psi$ to be any scalar multiple of $cI-\Phi^\dagger\Phi$, which in turn allows for the discriminants to be equal.

\subsection{Tightening Equiangularity}

Now we wish to characterize the situations when $(a,b)$-equiangular systems of lines are ETFs. The following result is originally due to Gerzon,
but was generalized in the finite field setting by~\cite{greaves_frames_2021}.

\begin{thm}\label{thm:gerzons}(Gerzon's Bound. Theorem 4.2 \cite{greaves_frames_2021})
    Let $V$ be a non-isotropic space, with $d=\dim V$, $k=\dim_{\F_0}{\F} \in \{1,2\}$, and $a^2\neq b$, then 
    there exists an $(a,b)$-equiangular system of $n$ lines in $V$ only if 
    $n\leq d+\frac{k}{2}(d^2-d)$. In the case of equality there exists some $c\in \F_0$ 
    where $\Phi$ is an $(a,b,c)$-ETF for $V$.
\end{thm}

\cref{thm:gerzons} gives the \textbf{absolute bound} on the number of equiangular lines depending on the space.
Although no bound is known relative to the parameter $b$ we can still characterize the situations in which an $(a,b)$-equiangular systems of lines $\Phi$ is an ETF depending only on the parameters $a$, $b$ and the Gram matrix $\Phi^\dagger\Phi$.

Let $\Phi:\F^n\ra V$ be an $(a,b,c)$-ETF where $d=\dim(V)$. Looking at the traces of $\Phi^\dagger\Phi$ and 
$\Phi\Phi^\dagger$ we have the following relation:
\begin{equation}\label{eq:nadc}
    \tr(\Phi^\dagger\Phi)=na=dc=\tr(\Phi\Phi^\dagger).
\end{equation}
Looking at $(\Phi^\dagger\Phi-aI)^2$, \cite{greaves_frames_2021} showed that

\begin{equation}
    \label{eq:acan1b}
    a(c-a)=(n-1)b.
\end{equation}
When $\Char\F$ does not divide $d(n-1)$, putting \eqref{eq:nadc} and \eqref{eq:acan1b} together yields
\begin{equation}
    \label{eq:welch}
b=\frac{(n-d)}{d(n-1)}a^2.
\end{equation}
We note that $a^2(n-d)=d(n-1)b$ 
is always true regardless of the characteristic.
Unlike in the real or complex setting, satisfying this equality, known as the \textbf{Welch bound}, does not guarantee that an equiangular system of lines is an ETF.

\begin{ex}\label{ex:welchweird}
    Consider the following $(2,1)$-system of equiangular lines which is a frame for $\F_5^7$ in the real model
    \[\Phi=\begin{bmatrix}
        0 & 0 & 0 & 0 & 0 & 0 & 0 & 2\\  
        0 & 0 & 0 & 0 & 0 & 1 & 2 & 0\\  
        0 & 0 & 0 & 0 & 2 & 4 & 2 & 0\\  
        0 & 0 & 0 & 0 & 2 & 4 & 0 & 2\\  
        0 & 1 & 1 & 2 & 1 & 2 & 2 & 3\\  
        1 & 0 & 1 & 2 & 3 & 2 & 2 & 3\\  
        1 & 1 & 0 & 2 & 3 & 4 & 4 & 1
    \end{bmatrix}\]
    which also satisfies $(n-1)b\equiv 2\equiv \frac{n-d}{d}a^2$. However, $\Phi$ is not a tight frame since $\Phi\Phi^\dagger$ is not diagonal.
\end{ex}

Under additional constraints, i.e., concerning sums of triple products of the vectors in $\Phi$, we can come up with a sufficient condition for being an ETF. Sums of triple products have been leveraged previously to understand the algebraic structure of ETFs over characteristic zero \cite{king_2-_2019,Zhu15}. If $\Phi=(\phi_j)_{j=1}^n$ is an 
$(a,b,c)$-ETF for $V$, where $\dim V=d$ we have that 
\begin{equation}
    \label{eq:sumoftripleprods}
    \sum_{\ell=1}^n\Delta(\phi_j,\phi_k,\phi_\ell)=\ip{\phi_j}{\phi_k}    \ip{\phi_k}{\sum_{\ell=1}^n\ip{\phi_\ell}{\phi_j}\phi_\ell} =\ip{\phi_j}{\phi_k} \ip{\phi_k}{c\phi_j}=cb.
\end{equation}
Here we will create a non-degenerate scalar product on linear operators and use non-degeneracy to show that an operator is the zero operator because the scalar product with every other operator is $0$. 
\begin{defn}
    Let $L=\{A:V\ra V\}$ be the $\F$-vector space of linear operators on a non-isotropic space $V$. Under a choice of basis we can consider $L$ to be the space of $d\times d$ matrices where $\dim(V)=d$. The \textbf{Frobenius scalar product} is then defined for $A,B \in L$ as 
    \[\ip{A}{B}_F:=\tr(A^\dagger B).\] 
\end{defn}

We note that in the case where $V=\F^n$ with the standard dot product or conjugate dot product, the adjoint operator on any map $A:\F^n\ra\F^n$ is the conjugate transpose and the matrices $\set{E_{ij}}{i,j\in[n]}$ form an orthonormal basis for $L$ with respect to the Frobenius scalar product, as $\tr(E_{ij}E_{\ell k})=1$ when $j=\ell$ and $i=k$ and zero otherwise. 
Therefore, the Frobenius scalar product is a non-degenerate Hermitian scalar product on $L$.

We may also consider the subspace of self adjoint operators $L_0=\set{A\in L}{A^\dagger=A}$ which is an $\F_0$-vector space. Let $V=\F_q^n$ be in the real model for case O, in which case 
\[
\set{E_{\ell\ell}}{\ell \in [n]} \cup \set{E_{ij}+E_{ji}}{i,j\in[n], i<j} 
\]
forms an orthogonal basis for $L_0$,
Notice that $\ip{E_{\ell\ell}}{(E_{ij}+E_{ji})}_F = 0$. Looking at the symmetrical terms we see that $\ip{(E_{ij}+E_{ji})}{(E_{\ell k}+E_{k\ell})}_F=2$ when $i=\ell$ and $j=k$ and zero otherwise. This means that the Frobenius scalar product is a non-degenerate symmetrical scalar product on $L_0$ in the real model, but in general and particularly in case U, the Frobenius scalar product may be degenerate.

\begin{thm}\label{thm:theresult}
    Let $\Phi=(\phi_j)_{j=1}^n$ be an $(a,b)$-equiangular system in a non-isotropic space $V$ where 
    $\dim(\image\Phi)=d$, such that $\Char\F>d$ 
    and $\frac{na}{d}\neq 0$.
    Then $\Phi:\F^n\ra \image\Phi$ is an $(a,b,\frac{n}{d}a)$-ETF if and only if
$(n-1)b=\frac{n-d}{d}a^2$ and
    $\sum_{\ell=1}^n\Delta(\phi_j,\phi_k,\phi_\ell)=\frac{nab}{d}$
    for all $j\neq k$ in $[n]$.
    \end{thm}
\begin{proof}
The forward implication immediately follows from \eqref{eq:welch} and \eqref{eq:sumoftripleprods}.
    For the other implication, it suffices to show that $(\Phi^\dagger \Phi)^2=\frac{na}{d}(\Phi^\dagger \Phi)$, because when
    $\Char\F>d$, \cref{thm:buildingcharacter} gives us that $\image\Phi$ is non-isotropic and so 
    $\Phi:\F^n\ra \image \Phi$ would be a $\frac{na}{d}$-tight frame.

    Because $\F^n$ is non-isotropic with a non-degenerate scalar product, we know that the Frobenius scalar product on 
    $\F^{n\times n}$ is a non-degenerate Hermitian scalar product.
    This means we can show that $(\Phi^\dagger \Phi)^2-\frac{na}{d}(\Phi^\dagger \Phi)=0$ by showing that 
    $\ip{(\Phi^\dagger \Phi)^2-\frac{na}{d}(\Phi^\dagger \Phi)}{A}_F=0$ for all $A\in\F^{n\times n}$,
    or likewise showing that
    $\ip{(\Phi^\dagger \Phi)^2-\frac{na}{d}(\Phi^\dagger \Phi)}{E_{ij}}_F=0$ for all $i,j$.
    We will look at two cases, first when $i=j$ and then when $i\neq j$.
    \begin{align*}
        \ip{(\Phi^\dagger\Phi)^2 -\frac{na}{d}\Phi^\dagger\Phi}{E_{jj}}_F
        &=\ip{(\Phi^\dagger\Phi)^2}{ E_{jj}}_F-\frac{na}{d}\ip{\Phi^\dagger\Phi}{E_{jj}}_F\\
        &=\tr((\Phi^\dagger\Phi)^2 E_{jj})-\frac{na}{d}\tr(\Phi^\dagger\Phi E_{jj})\\
        &=\sum_{k=1}^n\ip{\phi_j}{\phi_k}\ip{\phi_k}{\phi_j}-\frac{na}{d}\ip{\phi_j}{\phi_j}\\
        &=a^2+(n-1)b-\frac{na^2}{d}\\
        &=(n-1)b-\frac{(n-d)a^2}{d}=0,
    \end{align*}
    where the last equality follows from the assumption $(n-1)b=\frac{n-d}{d}a^2$.
    We also need to consider the elements $E_{ij}$ where $i\neq j$, and assume that $b\neq 0$
    \begin{align}
        \ip{(\Phi^\dagger\Phi)^2 -\frac{na}{d}\Phi^\dagger\Phi}{E_{ij}}_F
        &=\ip{(\Phi^\dagger\Phi)^2}{E_{ij}}_F-\frac{na}{d}\ip{\Phi^\dagger\Phi}{E_{ij}}_F \nonumber\\
        &=\tr((\Phi^\dagger\Phi)^2 E_{ij})-\frac{na}{d}\tr(\Phi^\dagger\Phi E_{ij}) \nonumber\\
        &=\sum_{k=1}^n\ip{\phi_j}{\phi_k}\ip{\phi_k}{\phi_i}-\frac{na}{d}\ip{\phi_j}{\phi_i} \label{eqn:zerook}\\
        &=\frac{1}{\ip{\phi_i}{\phi_j}}\left(\sum_{k=1}^n\Delta(\phi_j,\phi_k,\phi_i)-\frac{na}{d}b\right) \nonumber\\
        &= 0 \nonumber,
    \end{align}
    where the last equality follows from $\sum_{k=1}^n\Delta(\phi_j,\phi_k,\phi_i)=\frac{nab}{d}$.
    Notice also that if $b=0$, \eqref{eqn:zerook} would be $0$, and so we would get the same result.
    This gives us $(\Phi^\dagger \Phi)^2=\frac{na}{d}(\Phi^\dagger \Phi)$ as desired. And so by \cref{thm:buildingcharacter} we have that $\Phi:\F^n\ra \image\Phi$ is an $(a,b,\frac{na}{d})$-ETF.
\end{proof}

\begin{thm}\label{thm:theresult2}
        Let $\Phi=(\phi_j)_{j=1}^n$ be an $(a,b)$-equiangular system which is also a frame for a non-isotropic $d$-dimensional space $V$, such that $\Char\F\nmid d$ 
    and $\frac{na}{d}\neq 0$.
    Then $\Phi:\F^n\ra V$ is an $(a,b,\frac{n}{d}a)$-ETF if and only if
$(n-1)b=\frac{n-d}{d}a^2$ and
    $\sum_{\ell=1}^n\Delta(\phi_j,\phi_k,\phi_\ell)=\frac{nab}{d}$
    for all $j\neq k$ in $[n]$.
\end{thm}
\begin{proof}
    In this case because $\Phi$ is a frame for $V$, $(\Phi^\dagger \Phi)^2=\frac{na}{d}(\Phi^\dagger \Phi)$ would imply that $\Phi$ is also a $\frac{na}{d}$-tight frame. The proof is then nearly identical to \cref{thm:theresult} and so we have that $\Phi:\F^n\ra V$ is an $(a,b,\frac{na}{d})$-ETF.
\end{proof}

Now we want to look at one consequence of \cref{lem:totally_iso_ngeq2d} and \eqref{eq:welch} which was 
originally noted in Remark 2.15 in \cite{greaves_frames_2021-1}.
\begin{rem}
    \label{rem:bneq0}
Consider an $(a,b,c)$-ETF $\Phi=(\phi_j)_{j=1}^n$ for a $d$ dimensional space where $n>d$. If $b=0$, then with the relation $a(c-a)=(n-1)b$ 
this must imply that either $a$ or $c-a$ are zero. If $a$ is zero, then the corresponding Gram 
matrix $G=0$, which is not possible by \cref{lem:nondegiff}, so $a$ can not equal $0$, when $b=0$.
Now we will consider the case where $c-a=0$, or in other words $c=a\neq 0$. We also know that $na=dc$
which suggests that $n=d$. So when $n>d$ we can conclude that $b\neq 0$.
\end{rem}

\subsection{Two-graphs and ETFs in Orthogonal Geometries}

Looking specifically at orthogonal geometries, we can determine many analogous combinatorial equivalent notions to tightness.
The following is a corollary of Theorem 4.3 from \cite{greaves_frames_2021-1} which generalizes some of the results from \cite{seidel_two_graphs_76,waldron_equiangular_from_graphs_2009}.

\begin{thm}
    \label{lem:etfiffregulartwograph}
    Fix a prime $p>n$.
    Let $\Phi$ be an $(a,1)$ equiangular system of $n$ vectors which is also a frame for a $d$-dimensional orthogonal space $V$ over the field $\F_{p^\ell}$ whose induced two-graph is non-trivial and $n>d$.
    $\Phi$ is an $(a,1,c)$-ETF for $V$
    if and only if the induced two-graph is regular.
\end{thm}
\begin{proof}
    Fix a prime $p>n$. Let $\Phi=(\phi_j)_{j=1}^n$ be an $(a,1)$-equiangular frame for $V=\F_{p^\ell}^d$.\\
    ($\Rightarrow$)
    Let $\Phi$ be an $(a,1,c)$-ETF, then $\Phi$ is switching equivalent to some ETF $\Psi$ such that
    \[
    \Psi^\dagger\Psi=\begin{bmatrix}
        a & (1^{n-1})^\intercal\\
        1^{n-1} & aI_{n-1}+\overline\Sigma 
    \end{bmatrix}
    \] 
    where $1^{n-1}$ is the all ones vector with $n-1$ entries, and $\Sigma$ is the adjacency matrix for $G_{\phi_1}$ with $v=n-1$ vertices.
    From Theorem 4.3 of \cite{greaves_frames_2021-1} we know that $\Sigma$ is the adjacency matrix of $G_{\phi_1}$ which is a $(n-1,k,\lambda,\mu)$-$\text{SRG}_p$, with parameter chosen minimally, satisfying $k\equiv_p 2\mu$, $v\equiv_p 3k-2\lambda -1$ and that there exists some $\delta\in\F_{p^{\ell'}}$ such that $\delta^2\equiv (\lambda-\mu)^2+4(k-\lambda)$. \\ 
    Notice that if $\Char\F=p>n$ we would have that $G_{\phi_1}$ is an $(n-1,k,\lambda,\mu)$-SRG, and because $k-2\mu\leq n-2\mu\leq n<p$ we have that
    $k=2\mu$. And so $\Phi$ induces a regular two-graph.\\
    ($\Leftarrow$)
    Assume that $\Phi$ is an $(a,1)$-equiangular system of lines which induces a regular two-graph with parameters $(n,\ell, m)$ which satisfy $n=3m-2\ell$. Because the induced two-graph is regular there exists a vector $\phi_1$, such that $G_{\phi_1}$ is strongly regular, meaning there is a switching equivalent system of equiangular lines $\Psi$ where 
    \[
    \Psi^\dagger\Psi=\begin{bmatrix}
        a & (1^{n-1})^\intercal\\
        1^{n-1} & aI_{n-1}+\overline\Sigma 
    \end{bmatrix}
    \]
    such that $\Sigma$, as an integer matrix, is the adjacency matrix of $G_{\phi_1}$ which is a $(v:=n-1,k:=m,\lambda:=\ell,\mu:=\frac{m}{2})$-SRG with $k=2\mu$, $v= 3k-2\lambda -1$ and $(\lambda-\mu)^2+4(k-\lambda)=(\ell-\frac{m}{2})^2+4(m-\ell)$ being the square of an integer. This means that Theorem 4.3 from \cite{greaves_frames_2021-1} implies that $\Phi$ is a $(a,1,c)$-ETF where $c$ is a square root of $(\ell-\frac{m}{2})^2+4(m-\ell)$. 
\end{proof}
We note that a similar result holds for $b \neq 1$, but requires passing to field extensions and rescaling.

\begin{cor}
    \label{cor:niseven}
    Let $\Phi$ be a $(a,b)$ equiangular system of $n$ lines in a $d$-dimensional orthogonal space $V$ over the field $\F_{p^\ell}$ where $p>n$,
    then $\Phi$ is an $(a,b,c)$-ETF for $V$
    only if $n$ is even
\end{cor}
\begin{proof}
    This follows from the fact that regular two-graphs have an even number of points~\cite{Cameron_Lint_1991}.
\end{proof}

\section{Simplices: Equiangular Cliques}\label{sec:simplex}
In this section and in \cref{ssec:incoherentreality} we generalize concepts like the binder~\cite{FICKUS201898}, pillars~\cite{LEMMENS1973494}, and incoherent sets~\cite{gillespie-2018-equiangular} by looking at collections of vectors whose triple products are equal. These correspond to the cliques and cocliques of two-graphs induced by equiangular systems of lines.

In this section we wish to study the minimal dependent subsets of frames and when they are themselves tight frames for their spans, 
and more specifically when they are ETFs. This has been studied for real and complex frames in \cite{FICKUS201898}.
\begin{defn}
A collection of $s+1$ vectors $\Phi=(\phi_j)_{j=1}^{s+1}$ for $s$ a positive 
integer is called a \textbf{regular $s$-simplex}
if $\Phi:\F^{s+1}\ra V$ is an $(a,b,c')$-ETF for an $s$-dimensional non-isotropic space $V$.
\end{defn}
We use the notation $c'$ with the prime as we will mainly explore when regular $s$-simplices are subsets of an $(a,b,c)$-ETF.

In this section we will assume any field $\F=\F_q$ is a finite 
field of case U or O where $q$ is odd, as this will allow us to use \cref{thm:naimarkexistsU,thm:naimarkexistsO} which will be necessary for the following analysis. We will also assume that $s>1$, in which case $b\neq 0$ by \cref{rem:bneq0}.
Existence of simplices in certain dimensions depends only on the characteristic of the field.

Let $\Phi=(\phi_j)_{j=1}^{s+1}$ be a regular $s$-simplex, an $(a,b,c')$-ETF for $V$, with $s>1$; then, we know that $c'\neq 0$
from \cref{lem:totally_iso_ngeq2d} as $s+1<2s$. This means there exists a Naimark complement 
$C:\F^{s+1}\ra W$ where $W=\F$ is a $1$-dimensional space and $C$ is a $(c'-a,(c'-a)^2,c')$-ETF by \cref{lem:naimarksarecomplentary} with $\discr(W)=(c')^{s+1}\discr(V)$.

The discriminant then determines the geometry and therefore the scalar product on $W=\F$. We will denote the scalar product on $W$ as $x\cdot_m y:=\ip{x}{y}_\F=x^\sigma m y$ where we can choose $m=1$ if and only if $\discr(W)$ is trivial. If the discriminant is non-trivial then $m$ must be some non-square element of $\F_0$. This makes $\cdot_m$ an example of an isotopy. 
In case U, we may always assume that $m=1$ up to isomorphism, and in case O we may assume that $x\cdot_m y=xym$ where $m\in\F_0$ is possible a non-square.
Because $C=\{c_j\}_{j=1}^{s+1}$ is a collection of non-zero constants we have that $c_j\cdot_m c_j=c_j^\sigma m c_j=c'-a$, which is a square if and only if the discriminant of $W$ is trivial. This means that $\discr(V)=(c'-a)(c')^{s+1}\F^{\times2}$ which follows from \cref{thm:naimarkexistsO}.
Likewise
\[b=(c_j\cdot_m c_k) (c_k\cdot_m c_j)=c_j^\sigma m c_kc_k^\sigma m c_j=c_j^\sigma m c_jc_k^\sigma m c_k=(c'-a)^2.\]
In this case we have that $(s+1)(c'-a)=c'$ which rearranges to give $a=s(c'-a)$, meaning that $a=0$ if and only if $\Char \F$ divides $s$.
We can also directly compute $c'=\sum_{j=1}^{s+1}c_j^\sigma m c_j=(s+1)(c'-a)$ meaning such a Naimark complement
could only exist when $\Char\F$ does not divide $s+1$ as $c'\neq 0$. This proves the \textit{only if} direction of the 
following lemma.

\begin{lem}(Example 2.16 in \cite{greaves_frames_2021-1})
    \label{lem:simplexexistchar}
    Fix a finite field $\F$ (in Case O or U). A regular $s$-simplex $\Phi:\F^{s+1}\ra V$ for some $s$-dimensional space $V$
    exists if and only if $\Char\F$ does not divide $s+1$.
\end{lem}
\begin{proof}
    All that remains is to prove the \emph{if} direction. Assume that $\Char\F$ does not divide $s+1$ for some fixed integer $s$.
    In this case we can construct a collection of constants $C=(c_j)_{j=1}^{s+1}$ such that $a:=c_j^\sigma c_j\neq 0$ 
    and is equal for all $j$. This means $C$ is an $(a,a^2, (s+1)a)$-ETF. By construction $(s+1)a\neq 0$; so, there 
    exists a Naimark complement which is an $((s+1)a-a,a^2,(s+1)a)$-ETF $\Phi:\F^{s+1}\ra V$ where $\dim(V)=s$. Thus, there exists
    a regular $s$-simplex.
\end{proof}

\begin{lem}\label{lem:simplexfull}
    Let $\Phi=(\phi_j)_{j=1}^{s+1}$ be a regular $s$-simplex for $V$ an 
    $s$ dimensional space. Then
    the vectors are minimally dependent; i.e., they are dependent but any proper subset is independent.
\end{lem}
We could prove the lemma using matroid theory, but we use the results from this paper.
\begin{proof}
By the work above, any simplex $\Phi$ has a Naimark complement $\Psi$ which is a sequence of scalars with the same modulus.  Applying $\image\Psi^\dagger = (\image\Phi^\dagger)^\perp$ from \cref{lem:naimarksarecomplentary}, we see that any linear combination of the columns of $\Phi$ that is equal to $0$ has coefficients a scaling of the elements of $\Psi$.  Thus, there is no way to make a linear combination with a non-empty strict subset of the $s+1$ vectors in $\Phi$ be equal to zero.
\end{proof}

Let $\Phi:\F^n\ra W$ be an $(a,b,c)$-ETF for some $d$-dimensional space. We say a collection of $s+1$ vectors
$\Phi|_\kappa:=(\phi_j)_{j\in\kappa}$ where $\kappa\se[n]$($|\kappa|=s+1$) 
is a \textbf{sub-ETF} if it forms an $(a,b,c')$-ETF for $\image\Phi|_\kappa$. The existence of sub-ETFs in SICs are analyzed in~\cite{ABDF17,DBBA2013} to make progress on solving Zauner's conjecture.
Furthermore, if $\image\Phi|_\kappa$ is $s$-dimensional then we say $\Phi$ contains the 
regular $s$-simplex $\Phi|_{\kappa}:=(\phi_j)_\kappa$ in which case $b=(c'-a)^2$.
Notice that because $a$ and $b$ are equal in the entire frame $\Phi$ and in any regular 
$s$-simplex, we have that $(s+1)a=sc'$ where $c'\neq 0$,
and multiplying by $n$ we get $(s+1)dc=nsc'$.
Because a simplex is an ETF we also have that $a(c'-a)=sb$ and so by multiplying by $s$
we get that $a^2=s^2b$.
\begin{ex}
    Consider the unitary geometry $V=\F_{5^2}^3$ where $\F_{5^2}=\F_5[\alpha]/(\alpha^2+\alpha+1)$ with the standard Hermitian scalar product. Consider the following frame for $V$:
    \[
    \Phi=\begin{bmatrix}
        1  & 1     &  1  & 4 & 4 & 4   &  0 &  0   & 0\\       
        0  & 0     &  0  & 1  & \alpha & \alpha^2 & 4 & 4\alpha  & 4\alpha^2\\
        4 & 4\alpha^2 & 4\alpha & 0  & 0  & 0    & 1  & \alpha^2 & \alpha \\
    \end{bmatrix},
    \]
    where $\Phi$ is a $(2,1,1)$-ETF. Then $\Phi$ contains $12$ regular $2$-simplices which correspond to the Hesse configuration, which encodes the affine geometry of $\bF_3^2$.  There are no other simplices of any other size. $\Phi$ is a finite field analog of the Hesse SIC for $\C^3$~\cite{Hugh07}.
\end{ex}

\begin{prop}
    Fix $s>1$
    Let $\Phi=(\phi_j)_{j=1}^n$ be an $(a,b,c)$-ETF for $V$, a $d$-dimensional non-isotropic space. $\Phi$ contains a 
    regular $s$-simplex $\Phi|_\kappa:=(\phi_j)_{j\in\kappa}$, an $(a,b,c')$-ETF for $\image\Phi|_{\kappa}$ only if
    $a^2=s^2b$. 
    This equation is equivalent to $\frac{a^2}{b}=s^2$ and implies $(n-d)s^2\equiv d(n-1)$.

    \noindent If $\Char\F$ does not divide $sd(n-1)$ then $\frac{n-d}{d(n-1)}\equiv \frac{1}{s^2}$
\end{prop}
\begin{proof}
    We have already shown that if a regular $s$-simplex exists it would satisfy $\frac{a^2}{b}=s^2$.
    Notice that this in combination with the 
    parameters of the original frame which satisfy $(n-d)a^2=d(n-1)b$ would give us 
    $(n-d)s^2\equiv d(n-1)$.
    If $\Char\F$ does not divide $sd(n-1)$ we can 
    rearrange this to get $\frac{n-d}{d(n-1)}=\frac{1}{s^2}$.
\end{proof}

This also suggests that in the case where $a=0$, which could only happen when $\Char\F$ divides $s$, we would also have that $\Char\F$ divides $d$ or $n-1$.

\subsection{ETFs with Respectable Characteristics}
Throughout this section we will assume that $\Char\F\nmid s$ and therefore any regular simplex would have $a\neq 0$. 
Let $\Phi=(\phi_j)_{j=1}^n$ be an $(a,b,c)$-ETF for a $d$-dimensional non-isotropic space $V$, and assume that $\Phi$ contains a 
regular $s$-simplex $(\phi_j)_{j\in\kappa}$. In this case we have $\frac{s+1}{s}a=c'$ and thus $c'-a=\frac{a}{s}$. From before we have $\discr(V)=(c'-a)(c')^{s+1}\F^{\times 2}$, which means
if $s$ is odd then $(c')^{s+1}$ is a square meaning $\discr(\image\Phi|_\kappa)=\frac{a}{s}\F^{\times 2}$. If $s$ is even then $(c')^{s+1}$ is a square if and only if $c'$ is. Meaning $\discr(\image\Phi|_\kappa)=(s+1)\F^{\times 2}$.
More generally, \[\discr(\image\Phi|_\kappa)=\left(\frac{a}{s}\right)^s(s+1)^{s+1}\F^{\times 2}.\]

Over characteristic zero, saturating the Welch bound is sufficient to be an ETF.  We showed in \cref{thm:theresult} that the critical additional information needed over positive characteristic involves sums of triple products.  We have a similar result for when an ETF contains a regular $s$-simplex, which may be seen as a corollary to  \cref{thm:theresult}.

\begin{cor}
    Consider a field $\F$ such that $\Char\F>s+1$. Let $\Phi=(\phi_j)_{j=1}^n$ be an $(a,b,c)$-ETF for a $d$-dimensional non-isotropic space $V$. $\Phi$ contains a 
    regular $s$-simplex $\Phi|_{\kappa}:=(\phi_j)_{j\in\kappa}$ if and only if $|\kappa|=s+1$, $a^2=s^2b$, $\dim(\image\Phi|_{\kappa})=s$ and
    \begin{equation}
        \label{eq:2.21.4.5.6455}
        \sum_{\ell\in\kappa}\Delta(\phi_j,\phi_k,\phi_\ell)=\frac{(s+1)ab}{s}=\frac{(s+1)a^3}{s^3}
    \end{equation}
    for all $j\neq k$ in $\kappa$
\end{cor}
\begin{proof}
    ($\Rightarrow$) First, we assume that $\Phi$ contains a regular $s$-simplex $\Phi|_{\kappa}:=(\phi_j)_{j\in\kappa}$. From the 
    work above we know that $|\kappa|=s+1$ where $a^2=s^2b$, and $\dim(\image\Phi|_{\kappa})=s$. Likewise because 
    $(\phi_j)_{j\in\kappa}$ is an $(a,b,\frac{s+1}{s}a)$-ETF we know that 
    \[\sum_{\ell\in\kappa}\Delta(\phi_j,\phi_k,\phi_\ell)=\frac{(s+1)ab}{s}\]
    for all $j\neq k$ in $\kappa$.

    \noindent ($\Leftarrow$) Now consider a sub-collection of vectors $\Phi|_{\kappa}=(\phi_j)_{j\in\kappa}$ 
    where $|\kappa|=s+1$, $a^2=s^2b$,
    $\dim(\image \Phi|_{\kappa})=s$, and \eqref{eq:2.21.4.5.6455} is satisfied. Then by 
    \cref{thm:theresult} we have that $(\phi_j)_{j\in\kappa}$ is an ETF and therefore a regular $s$-simplex
\end{proof}

We note that a similar corollary follows from \cref{thm:theresult2}.
Let $\Phi=(\phi_j)_{j=1}^{s+1}$ be a regular $s$-simplex, an $(a,b,c')$-ETF, with $s>1$. 
As before we have a Naimark complement $C:\F^{s+1}\ra \F$, 
a $(c'-a,(c'-a)^2,c')$-ETF.
Because $C$ is a collection of constants we have that that the triple products of $C$ satisfy 
\[(c_j\cdot_m c_k)(c_k\cdot_m c_\ell)(c_\ell\cdot_m c_j)=c_j^\sigma m c_kc_k^\sigma m c_\ell c_\ell^\sigma m c_j=c_j^\sigma m c_jc_k^\sigma m c_kc_\ell^\sigma m c_\ell=(c'-a)^3.\]
Likewise, because $\Phi^\dagger \Phi= c'I-C^\dagger C$ we know that
$\ip{\phi_j}{\phi_k}=-(c_j\cdot_m c_k)$ for $j\neq k$ and $\ip{\phi_j}{\phi_j}=c'-(c_j\cdot_m c_j)=c'-(c'-a)=a$. 
Putting these together we can see that for distinct $j,k,\ell$ we have
\[\Delta(\phi_j,\phi_k,\phi_\ell)=-(c'-a)^3\]
where $c'-a\neq 0$.
If $\Char\F$ does not divide $s$ we know that $\Phi$ is an $(a,b,\frac{s+1}{s}a)$-ETF where $-(c'-a)^3=-\left(\frac{s+1}{s}a-a\right)^3=-\frac{a^3}{s^3}$. So if an ETF contains a regular $s$-simplex, 
then the triple products of distinct vectors of the simplex would all be equal to $-(c'-a)^3=-\frac{a^3}{s^3}$.
Similar to the real and complex cases, shown by \cite{FICKUS201898}, we will see that simplices are more or less determined 
by their triple products all being equal in this way, 
with a few annoying caveats.

\begin{thm}
    \label{thm:iffsimplex}
    Let $\Phi=(\phi_j)_{j=1}^n$ be an $(a,b,c)$-ETF for $d$-dimensional space $V$ over the field $\F_q$, 
    such that $d<n$ and $\Char\F_q$ does not divide $s(s+1)$.
    Then $\kappa\se [n]$ of size $s+1$ gives a regular $s$-simplex $(\phi_j)_{j\in\kappa}$ if and only if 
    $a^2=s^2b$,
    \[\Delta(\phi_j,\phi_k,\phi_\ell)=-\frac{a^3}{s^3}\neq 0\]
    for all distinct $j,k,\ell\in \kappa$, and 
    \[\sum_{j\in\kappa}\Delta(\phi_\ell,\phi_k,\phi_j)=\frac{s+1}{s}ab\]
    for a fixed $\ell\in\kappa$ and all $k\not\in\kappa$.
\end{thm}
\begin{proof}
    We have already shown most of the ($\Rightarrow$) direction; so, we will only show the ($\Leftarrow$) direction. Assume that $\Delta(\phi_j,\phi_k,\phi_\ell)=-\frac{a^3}{s^3}\neq 0$ for all distinct $j,k,\ell\in \kappa$.
    Pick $\alpha \in\F_q^\times$ and $m\in\F_0^\times$ such that $\alpha^\sigma m \alpha =\frac{a}{s}$. Notice that $m$ would be a square depending on if $\frac{a}{s}$ is, and if $\frac{a}{s}$ was a square then we could assume that $m=1$, otherwise $m$ would be a non-square. 
    Fix some $\ell\in \kappa$ and define $C=\{c_j\}_{j\in\kappa}$ such that 
    $c_\ell=\alpha$ and 
    $c_j=\frac{-1}{m\alpha^\sigma}\ip{\phi_{\ell}}{\phi_j}$
    for all $j\neq \ell$. Notice that 
    \[
        c_j^\sigma m c_j=\left. \begin{cases}
            (-1)^2\frac{s}{a}\ip{\phi_{j}}{\phi_{\ell}}\ip{\phi_{\ell}}{\phi_j} & \text{if }j\neq \ell\\
            \alpha^\sigma m \alpha & \text{if }j=\ell\\
        \end{cases}\right\}=\frac{a}{s}
    \]
    and likewise for $j\neq k$
    \[
        c_j^\sigma m c_k=\left. \begin{cases}
            -\ip{\phi_{\ell}}{\phi_k} & \text{if }j= \ell\\
            -\ip{\phi_j}{\phi_{\ell}} & \text{if }k= \ell\\
            \frac{s}{a}\ip{\phi_j}{\phi_{\ell}}\ip{\phi_{\ell}}{\phi_k} & \text{otherwise}\\
        \end{cases}\right\}=-\ip{\phi_j}{\phi_k},
    \]
    where the last line follows from $-\frac{s}{a}\ip{\phi_j}{ \phi_{\ell}}\ip{\phi_{\ell}}{\phi_k}\ip{\phi_k}{\phi_j}=\frac{a^2}{s^2}=b$.
    This shows us that 
    \[\Phi|_\kappa^\dagger\Phi|_\kappa=\frac{s+1}{s}aI-C^\dagger C.\]
    Because $C$ is a sequence of constants it is trivially an $(\frac{a}{s},b,\frac{s+1}{s}a)$-ETF, meaning $C^\dagger C$ has one non-zero eigenvalue 
    equal to $\frac{s+1}{s}a$, with corresponding eigenvector $C^\dagger$.
    From this we can determine that the $\rk(\Phi|_\kappa^\dagger\Phi|_\kappa)=s$.  
    We want to show that $\Phi|_\kappa$ is a Naimark complement of $C$ and so we need to show that $\Phi|_{\kappa}C^\dagger=0$. We being by computing
    
    \[\Phi|_{\kappa}C^\dagger=\sum_{j\in\kappa} c_j^\sigma m\phi_j
    =\sum_{j\in\kappa} c_j^\sigma m c_\ell\phi_j,
    \]
    and we will show that $\sum_{j\in\kappa} c_j^\sigma m c_\ell\phi_j=0$ using the 
    non-degeneracy of $V=\image(\Phi)$.
    So consider a vector $\phi_k\in \Phi$ and compute
    \begin{align*}
        \ip{\phi_k}{\sum_{j\in\kappa} c_j^\sigma m c_\ell\phi_j}&=\sum_{j\in\kappa}c_j^\sigma m c_\ell\ip{\phi_k}{\phi_j}\\
        &=c_\ell^\sigma m c_\ell\ip{\phi_k}{\phi_\ell} -\sum_{j\neq\ell}\ip{\phi_k}{\phi_j}\ip{\phi_j}{\phi_\ell}.
    \end{align*}
   This expression is zero if and only if the following is zero
    \begin{align*}
        c_\ell^\sigma m c_\ell\ip{\phi_\ell}{\phi_k}\ip{\phi_k}{\phi_\ell} -\sum_{j\neq \ell}\ip{\phi_\ell}{\phi_k}\ip{\phi_k}{\phi_j}\ip{\phi_j}{\phi_\ell}
        &=\frac{ab}{s} -\sum_{j\neq\ell}\Delta(\phi_\ell,\phi_k,\phi_j)\\
        &=\frac{ab}{s} -\sum_{j\in\kappa}\Delta(\phi_\ell,\phi_k,\phi_j)+ab \\
        &= 0,
    \end{align*}
    which follows from the initial assumption.
\end{proof}

\subsection{ETFs with Defiant Characteristics}
Now we want to explore the possibility of the characteristic dividing $s$, which leads to some very different behavior. Consider an $(a,b,c)$-ETF which contains a regular $s$-simplex, $\Phi|_\kappa$ which is an $(a,b,c')$-ETF, which means $\Char\F\nmid s+1$.
If $\Char\F|s$, then we would have $a=0$ and $b\neq 0$, and so $na=dc$ implies that $c=0$ or $\Char\F|d$.
Likewise from $a(c-a)=(n-1)b$ we know that $\Char\F | (n-1)$ as $b\neq 0$. Looking at the discriminant
we know that \[\discr(\image\Phi|_\kappa)=(c'-a)(c')^{s+1}\F^{\times2}=(c')^{s+2}\F^{\times2}=(c')^{s}\F^{\times2}.\]
We note $c'$ is intrinsic to the simplex and not the entire frame, meaning it is possible to have an ETF with multiple simplices which have different geometries.
\begin{ex}
    \label{ex:weridsimplices}
    Consider the orthogonal geometry on $\F_3^4$ with the scalar product whose Gram matrix is $\diag(1,1,1,2)$.
    Notice that $-1\equiv 2$ is not a square in $\F_3$ so this is an example of an orthogonal geometry that is not the real model.
    Now consider the frame for $\F_3^4$:
    \[\Phi = \begin{bmatrix}
        0 & 0 & 0 & 0 & 1 & 1 & 1 & 1 & 1 & 1\\  
        0 & 0 & 1 & 1 & 0 & 0 & 1 & 1 & 2 & 2\\  
        1 & 1 & 0 & 0 & 0 & 0 & 1 & 2 & 1 & 2\\  
        1 & 2 & 1 & 2 & 1 & 2 & 0 & 0 & 0 & 0\\
    \end{bmatrix},
    \]
    where $\Phi$ is an $(0, 1, 0)$-ETF of $n=10$ vectors.

    Let $\kappa=\{1,2,3,4\}$. Then $\Phi|_\kappa$ forms a regular $3$-simplex which is an $(0,1,2)$-ETF with $\discr(\image\Phi|_\kappa)=2\F^{\times 2}$.
    Now let $\overline{\kappa}=\{7,8,9,10\}$. Then $\Phi|_{\overline \kappa}$ also forms a regular $3$-simplex, but this one is a $(0,1,1)$-ETF in the real model where $\discr(\image\Phi|_\kappa)=\F^{\times 2}$. In fact there are $30$ regular $3$-simplices, $15$ of which are in the real model and the other $15$ which have a non-square discriminant. The $15$ simplices which are in the real model, whose indices are shown in \eqref{eq:realsimps}, form a $2$-$(10,4,2)$ design,
    
    \begin{align}
        \label{eq:realsimps}
        \begin{split}
            \{
            & \{1,2,7,8\},\{1,2,9,10\},\{1,3,5,7\},\{1,3,6,10\},\{1,4,5,9\},\\
            & \{1,4,6,8\},\{2,3,5,8\},\{2,3,6,9\},\{2,4,5,10\},\{2,4,6,7\},\\
            & \{3,4,7,9\},\{3,4,8,10\},\{5,6,7,10\},\{5,6,8,9\},\{7,8,9,10\}
            \}
        \end{split}
    \end{align}

    Likewise the $15$ simplices with non-square discriminant also form a $2$-$(10,4,2)$ design, whose blocks are shown in \eqref{eq:nonsquaresimps}.

    \begin{align}
        \label{eq:nonsquaresimps}
        \begin{split}
            \{
            & \{1,2,3,4\},\{1,2,5,6\},\{1,3,8,9\},\{1,4,7,10\},\{1,5,8,10\},\\
            & \{1,6,7,9\},\{2,3,7,10\},\{2,4,8,9\},\{2,5,7,9\},\{2,6,8,10\},\\
            & \{3,4,5,6\},\{3,5,9,10\},\{3,6,7,8\},\{4,5,7,8\},\{4,6,9,10\}\}
        \end{split}
    \end{align}
\end{ex}

\section{Connections to Reality}

In this section we draw connections between ETFs in orthogonal geometries with real ETFs.
There are no real ETFs of $n=10$ vectors in $\R^4$;
however, there is a an ETF of $n=10$ vectors in an orthogonal geometry on $\F_3^4$ (cf.\ \cref{ex:weridsimplices}). In this example we also know that in the field extension $\F_{3^2}=\F[x]/(x^2+1)$, $2$ becomes a square; so, the orthogonal geometry over $\F_{3^2}^4$ has a square discriminant.
In fact, \cite{greaves_frames_2021-1} showed that there are infinity many dimensions where maximal ETFs over orthogonal geometries are 
known but it is conjectured that such ETFs don't exist as real ETFs.

In \cite{greaves_frames_2021-1}
the authors showed that the existence of real ETFs imply the existence of ETFs in orthogonal geometries and showed the converse is also true when the characteristic is sufficiently large. 

\begin{thm}(Proposition 3.2 \cite{greaves_frames_2021-1})
    \label{lem:discret-reality}
    Suppose $G=S+aI$ is the Gram matrix of some real $d\times n$ ETF $\Phi$, where $S\in\Z^{n\times n}$ is the signature matrix. Fix a finite field $\F_q$, where $q=p^\ell$ is an odd prime power, with $\delta^2=n-1$ if $n=2d$. Then with
    \[a\equiv\begin{cases}
    \sqrt{\frac{d(n-1)}{n-d}} & n\neq 2d \\ \delta & n=2d\end{cases} \quad \textrm{and} \quad
    c\equiv\begin{cases}
        \sqrt{\frac{n^2(n-1)}{d(n-d)}} & n\neq 2d \\ 2\delta & n=2d\end{cases}
    \]
    the matrix $\bar{G}=\bar{S}+aI$ is the Gram matrix of an $(a,1,c)$-ETF $\Psi$ in an orthogonal geometry on $\F_q^{d'}$ where $d'\leq d$, with equality when $c\neq 0$.
\end{thm}

\begin{thm}(Proposition 3.3 \cite{greaves_frames_2021-1})
    If there exists an ETF of $n$ vectors in a $d$-dimensional orthogonal geometry over $\F_{q}$ where $q=p^\ell$ is odd and $p>2n-5$ then there exists a real ETF of $n$ vectors for $\R^d$ with the same $n$ and $d$.
\end{thm}
The requirement that $p>2n-5$ is not known to be tight, and in the remainder of this section we wish to look at the case where $d<p<2n-5$, drawing parallels to the real word.

\subsection{The Incoherence of a Finite Reality: Equiangular Cocliques}
\label{ssec:incoherentreality}
Here we will continue to explore orthogonal geometries drawing connections to what is known over $\R$. In \cite{gillespie-2018-equiangular}, Gillespie showed that equiangular lines in $\R^d$ with $n=d(d+1)/2$ vectors which saturate the incoherence bound exist if and only if $d=2,3,7,23$. This gives further evidence that those dimensions are the only dimensions in which there is an ETF of any type with  $n=d(d+1)/2$ vectors in $\bR^d$, which is the Gillespie conjecture. In this section we wish to complete a similar treatment of equiangular lines which saturate an analogous bound for orthogonal geometries.

\begin{defn}
    Let $\Phi$ be an $(a,b)$-equiangular system of vectors in a orthogonal geometry $V$, and fix some $\beta$ such that $\beta^2=b$. The set of vectors $\Phi$ is $\beta$-\textbf{incoherent} if $|\Phi|\leq 2$ or $\Delta(\phi_j,\phi_k,\phi_\ell)=\beta^3$ for all distinct $j,k,\ell$.
    We define the $\beta$-incoherence number $\Inc_\beta(\Phi)$ to be the size of the largest subset of $\beta$-incoherent vectors in $\Phi$.
    Likewise, $\Phi$ is $\beta$-\textbf{coherent} if $|\Phi|\geq 3$ and $\Delta(\phi_j,\phi_k,\phi_\ell)=-\beta^3$
\end{defn}

We note that the choice of $\beta$ will affect the incoherence number, which we highlight in the following example. Over $\R$, $\beta$ is chosen to be positive, and incoherent sets are linearly independent. In the finite field setting we will pick $\beta$, when possible, so that incoherent sets are linearly independent as well. 
\begin{ex}\label{ex:incohind}
    Consider $\F_{11}^2$ in the real model with the standard dot product and define
    \[\Phi=\begin{bmatrix}
        0 & 3 & 8\\
        1 & 5 & 5
    \end{bmatrix}\quad \textrm{and} \quad \Phi^\dagger\Phi = \begin{bmatrix}
        1 & 5 & 5\\
        5 & 1 & 5\\
        5 & 5 & 1
    \end{bmatrix}.\]
    $\Phi$ is a $(1,3,7)$-ETF. Notice that $5^2=3$ and that all three vectors of $\Phi$ are a $5$-incoherent set as $\Delta(\phi_1,\phi_2,\phi_3)=5^3=4$. However notice also that $6^2=3$, But $\Phi$ is not a $6$-incoherent set. So $\Inc_{5}(\Phi)=3$ and $\Inc_6(\Phi)=2$.
\end{ex}

\begin{rem}
    \label{rem:switchit}
    Assume that $\Phi$ is an $(a,b)$-equiangular system of lines which is also a $\beta$-incoherent set. We can construct a second system of $n$ equiangular lines $\Psi$ which is switching equivalent to $\Phi$. First let $\psi_1=\phi_1$. Then for each $1< j\leq n$, let $\psi_j=\phi_j$ if $\ip{\phi_1,\phi_j}=\beta$ and let $\psi_j=-\phi_j$ otherwise. By construction $\Psi=(\psi_j)_{j=1}^n$ is switching equivalent to $\Phi$. Because $\ip{\phi_j,\phi_k}\ip{\phi_k,\phi_\ell}\ip{\phi_\ell,\phi_j}=\beta^3$ for all distinct $j,k,\ell$, we know that $\ip{\psi_j,\psi_k}\ip{\psi_k,\psi_\ell}\ip{\psi_\ell,\psi_j}=\beta^3$; so, $\Psi$ is also an incoherent set.
    Notice also that because $\ip{\psi_1,\psi_j}=\beta$ for all $j>1$ and $\ip{\psi_j,\psi_k}\ip{\psi_k,\psi_1}\ip{\psi_1,\psi_j}=\beta^3$ we know that $\ip{\psi_j,\psi_k}=\beta$ for all distinct $j,k$. This shows that $\Psi$ is not only an incoherent set but $\ip{\phi_j,\phi_k}=\beta$ for all distinct $j,k$.
    We will regularly make the assumption that incoherent sets have all scalar products being equal, which is always possible up to switching equivalence.
\end{rem}

The definition of incoherent sets comes from the definition of an incoherent set of a two-graph from \cref{sec:combinatorics-two-graphs}. Let $\Phi$ be an $(a,b)$-equiangular system of lines, then the coherent triples, with respect to a choice of $\beta$ are the triples of vectors whose triple products are equal to $-\beta^3$.

\begin{lem}
    \label{lem:incoherencesetsarealmostnice}
    Let $\Phi$ be an $(a,b)$-equiangular system of lines in a $d$-dimensional non-isotropic orthogonal space $V$ such that 
    there exists a non-zero $\beta\in\F$ where $\beta^2=b$, $a\neq \beta$, and $\Phi$ is $\beta$-incoherent.
    The vectors of $\Phi$ are linearly independent or
    a minimally dependent set; therefore, $|\Phi|\leq d+1$.
    
    \noindent When $|\Phi| = d+1$ and $\Char\F\nmid d(d-1)$, then $a+d\beta=0$ and $\Phi$ is a regular $d$-simplex if $a\neq 0$.
    
    \noindent Furthermore if $\Char\F\nmid d(d-1)$ with $a\neq 0$ and $\beta\neq -\frac{a}{d-1}$ then if $|\Phi| = d$ the vectors of $\Phi$ are linearly independent. If the parameters also satisfy $\beta^3\neq -\frac{a^3}{d^3}$ then $|\Phi| \leq d$.
\end{lem}
\begin{proof}
    Assume that $|\Phi|=n$. We will construct a second system of $n$ equiangular lines $\Psi$ which is switching equivalent to $\Phi$ as in \cref{rem:switchit}.
    In this case, the Gram matrix is $G=\Psi^\dagger\Psi=\beta J+ (a-\beta)I$ where $J$ is the all-ones matrix.
    Notice that the all-ones vector $\mathbbm{1}$ is an eigenvector with eigenvalue $a+(n-1)\beta$. 
    Likewise we can consider the vectors whose entries sum to zero, which is an eigenspace of dimension $n-1$, with eigenvalue $a-\beta\neq 0$ by the initial assumptions. Notice that $\mathbbm{1}$ is in the eigenspace of vectors whose entires sum to zero if and only if $\Char\F$ divides $n$. In which case the $\dim(\ker\Psi)\leq \dim(\ker\Psi^\dagger\Psi)\leq 1$. 
    If $\Char\F$ does not divide $n$, then we have a basis for the domain of $G$ in terms of the eigenvectors.
    This means that $\ker\Psi\leq \ker\Psi^\dagger\Psi\leq \lspan\{\mathbbm{1}\}$ with $\ker\Psi^\dagger\Psi= \lspan\{\mathbbm{1}\}$ if and only if $a+(n-1)\beta=0$. This means that the vectors of $\Psi$ and therefore the vectors of $\Phi$ are linearly independent or form a minimally dependent set when $a+(n-1)\beta=0$.

    Now we wish to explore the maximal case. Assume that $n=d+1$, which would mean that $a+(n-1)\beta=a+d\beta=0$.
    This also means that $a^2=d^2b$, and $\Psi$ is a frame for $V$.
    Assume that $\Char\F$ neither divides $d=n-1$ nor $d+1=n$ in which case  $\frac{1}{d}a^2=db$.
    Likewise for any distinct $j,k$ we have that 
    \[\sum_{\ell=1}^n\Delta(\phi_j,\phi_k,\phi_\ell)=(n-2)\beta^3+2ab=(d-1)\beta^3+-2d\beta^3=-(d+1)\beta^3=\frac{d+1}{d}ab.\] 
    This means that $\Psi$ is an ETF, in particular a regular $d$-simplex by \cref{thm:theresult2} if $a\neq 0$.

    To prove the last claim we will first assume that $\Char\F\nmid d(d-1)$ and $a\neq 0$.
    Consider the case where $n=d$ in which case we need only show that $a+(n-1)\beta\neq0$. Notice that this is equivalent to $\beta\neq -\frac{a}{n-1}$.
    In this case if $n=d+1$, we would have that $\Phi$ is a regular $d$-simplex from above, but by \cref{thm:iffsimplex}, we would have that all triple products would equal $-\frac{a^3}{d^3}$ contradicting out assumption that all triple products equal $\beta^3\neq -\frac{a^3}{d^3}$. Thus, $n\leq d$.
\end{proof}
A field $\F_q$ with characteristic $p$ has a non-trivial third root of unity when $x^2+x+1$ has non trivial roots since $x^3-1=(x-1)(x^2+x+1)$. This is the case when $p\equiv 1 \pmod 3$ or $q=p^{2\ell}$. 

As a corollary of \cref{lem:incoherencesetsarealmostnice}, incoherent sets which are minimally dependent sets are simplices, and can be often easily ruled out.

\begin{cor}
    \label{lem:incoherencebdd}
    Let $\Phi$ be an $(a,b)$-equiangular system of $n$ lines in a $d$-dimensional non-isotropic orthogonal space $V$ where $a\neq 0$ and $a^2\neq b$. Further let $\beta\in\F$ where $\beta^2=b\neq 0$. 
    If $\Char\F\nmid d$ then
    \begin{equation}
        \label{eq:incohebdd}
        \min(\Inc_{\beta}(\Phi),\Inc_{-\beta}(\Phi))\leq d
    \end{equation}
    Furthermore if $\Char\F\nmid (d-1)$ and for $\beta$ such that $\Inc_{\beta}(\Phi)=d=\min(\Inc_{\beta}(\Phi),\Inc_{-\beta}(\Phi))$, then any maximal incoherent set $\Gamma$ is linearly independent.
\end{cor}
\begin{proof}
    We know that $\Inc_{\beta}(\Phi)\leq d+1$ and $\Inc_{-\beta}(\Phi)\leq d+1$.
    We will assume that the first bound is saturated in which case let $\Gamma$ be $d+1$ vectors of $\Phi$ which are $\beta$-incoherent. In this case we have that $a+d\beta=0$, so $a=-d\beta$.
    If the second bound is saturated there would be a collection of $d+1$ vectors $\Gamma'$ which are $-\beta$-incoherent. In this case we have that $a-d\beta=0$, so $a=d\beta$. For $a=d\beta=-d\beta$ we would need that $a=0$ which happens if and only if $\Char\F$ is $2$ or divides $d$. But from our assumptions, neither is the case. So both bounds cannot be saturated.

    Now assume that $\Char\F\nmid (d-1)$, $\Inc_{\beta}(\Phi)=d=\min(\Inc_{\beta}(\Phi),\Inc_{-\beta}(\Phi))$, and $\Gamma$ is a maximal $\beta$-incoherent set. There are two cases to consider here.
    First, if $\Inc_{-\beta}(\Phi)=d+1$, then the maximal $-\beta$-incoherent set is a regular $d$-simplex, meaning $\beta=\frac{a}{d}\neq -\frac{a}{d}$. So, the vectors of $\Gamma$ are linearly independent.
    Now assume that $\Inc_{\beta}(\Phi)=d=\Inc_{-\beta}(\Phi)$. In this case, assume that $\Gamma$ is a maximal $\beta$-incoherent set and $\Gamma'$ is a maximal $-\beta$-incoherent set, which is also a minimally dependent set.
    This means that $a-(d-1)\beta=0$ and therefore $\beta=\frac{a}{d-1}\neq -\frac{a}{d-1}$; so, $\Gamma$ must be linearly independent.
\end{proof}

We note that the bound in \eqref{eq:incohebdd} does not guarantee linear independence in general when either $\Inc_{\beta}(\Phi)=d$, or $\Inc_{-\beta}(\Phi)=d$ as highlighted in the following example.
\begin{ex}\label{ex:weridsimplices2}
    Consider again \cref{ex:weridsimplices}:
    \[\Phi = \begin{bmatrix}
        0 & 0 & 0 & 0 & 1 & 1 & 1 & 1 & 1 & 1\\  
        0 & 0 & 2 & 1 & 0 & 0 & 1 & 1 & 2 & 2\\  
        1 & 1 & 0 & 0 & 0 & 0 & 1 & 2 & 1 & 2\\  
        1 & 2 & 2 & 2 & 1 & 2 & 0 & 0 & 0 & 0\\
    \end{bmatrix},
    \]
    where $\Phi$ is a $(0, 1, 0)$-ETF of $n=10$ vectors for an orthogonal geometry on $\F_3^4$.
    Computationally we can determine that the first 4 vectors form a maximal set of $1$-incoherent vectors. And likewise the last 4 vectors form a maximal $-1$-incoherence set; however, neither are linearly independent.
\end{ex}

Let $\Phi:\F^n\ra V$ be an $(a,b)$-equiangular system of lines in a $d$-dimension orthogonal space $V$ where $\Char\F\nmid d(d-1)$, $a^2\neq b$ and $a\neq 0$. We will call the bound from \eqref{eq:incohebdd}
the \textbf{incoherence bound}, and we will denote $\Inc(\Phi):=\min(\Inc_{\beta}(\Phi),\Inc_{-\beta}(\Phi))$

For the remainder of this section we will be interested in systems of lines with maximal incoherent sets which are also linearly independent. \cref{lem:incoherencebdd} guarantees that when the incoherence bound is saturated there exists some $\beta$ such that if $\Gamma$ is a maximal $\beta$-incoherent set, then $\Gamma$ is linearly independent.

If $\Gamma$ is a maximal incoherent subset then by definition for any vector $\gamma$ not in $\Gamma$ there exists $\alpha_1,\alpha_2\in\Gamma$ where $\{\gamma,\alpha_1,\alpha_2\}$ is a coherent triple. We will define two sets 
\[\Gamma_i(\gamma)=\set{\delta\in\Gamma}{\{\gamma,\alpha_i,\delta\} \text{ is coherent}}\] 
for $i=1,2$. With \cref{rem:switchit} we can assume that this implies that $\ip{\gamma}{\alpha_1}=\pm \beta$ and $\ip{\gamma}{\alpha_1}=\mp \beta$ and the two sets above can be equivalently defined as 
\[\Gamma_1(\gamma)=\set{\delta\in\Gamma}{\ip{\gamma}{\delta}=-\ip{\gamma}{\alpha_1}} \quad \textrm{and} \quad \Gamma_2(\gamma)=\set{\delta\in\Gamma}{\ip{\gamma}{\delta}=-\ip{\gamma}{\alpha_2}},\]
or in other words, $\Gamma_1(\gamma)$ and $\Gamma_2(\gamma)$ partition $\Gamma$ based on the scalar products of elements in $\Gamma$ with $\gamma$, being $\pm\beta$, which are independent of the choice of $\alpha_1$ and $\alpha_2$. 
By convention we label the sets such that $|\Gamma_1(\gamma)|\leq |\Gamma_2(\gamma)|$. 

The following three lemmas are generalizations of Theorem 5.4, Theorem 5.6, and a remark in \cite{taylortwographs} which are used heavily in \cite{gillespie-2018-equiangular}.

\begin{lem}
    \label{lem:rootsofpolyarecardinality}
    Let $\Phi$ be an $(a,b)$-equiangular system of $n$ lines in a $d$-dimensional orthogonal geometry $V$ 
    such that $a^2\neq b$ and $a\neq 0$.
    Assume there exists $\Gamma\se \Phi$ that is a maximal $\beta$-incoherent set ($|\Gamma|=d$) that is also linearly independent. Then for every $\gamma$ not in $\Gamma$, we have that $|\Gamma_1(\gamma)|$ and $|\Gamma_2(\gamma)|$ are roots of the following equation
    \[4x^2-4dx+(\rho-1)^2(d+\rho)\equiv 0\]
    where $\rho=a\beta^{-1}$. Furthermore when $\Char\F>d$, the smallest integers which are roots are $|\Gamma_1(\gamma)|$ and $|\Gamma_2(\gamma)|$.
\end{lem}
\begin{proof}
    We know that the vectors of $\Gamma$ for a basis for $V$; denote them $\phi_1,\dots, \phi_d$ and consider the vector $\gamma\not\in \Gamma$. As in the proof of \cref{lem:incoherencesetsarealmostnice} we may pick a switching equivalent set of vectors $\{\psi_j\}_{j=1}^d$ which are a basis and have pairwise scalar products equal to $\beta$ by \cref{rem:switchit}.
    
    We will reorder the vectors of $(\psi_j)_{j=1}^d$, and rescale $\gamma$ by $-1$ such that $\phi_j\in\Gamma_1(\gamma)$ and $\ip{\gamma, \psi_j}=\beta$ for $1\leq j\leq |\Gamma_1(\gamma)|=:r$ and likewise for $\phi_j\in\Gamma_2(\gamma)$ and $\ip{\gamma}{\psi_j}=-\beta$ for $r<j\leq d$. Because the vectors $(\psi_j)_{j=1}^d$ are a basis we can write $\gamma=\sum_{j=1}^db_j\psi_j$. Notice that for $k\leq r$ we have that
    \[\ip{\gamma}{\psi_k}= \sum_{j=1}^db_j\ip{\psi_k}{\psi_j}=\sum_{j=1,j\neq k}^db_j\beta+b_ka =\beta \]
    and for $r<k\leq d$
    \[\ip{\gamma}{\psi_k}= \sum_{j=1}^db_j\ip{\psi_k}{\psi_j}=\sum_{j=1,j\neq k}^db_j\beta+b_ka=-\beta. \]
    This gives us a system of linear equations, and solving for the $b_j$'s
    we get 
    \begin{equation}\label{eq:gnarlybs}
    b_j=\begin{cases}\displaystyle{\frac{a\beta+(2d-2r-1)b}{(a+(d-1)\beta)(a-\beta)}}\\
        \displaystyle{\frac{(1-2r)b-a\beta}{(a+(d-1)\beta)(a-\beta)}} 
    \end{cases}=
    \begin{cases}
        \displaystyle{\frac{2d-2r+\rho-1}{(d+\rho-1)(\rho-1)}} & 1\leq j\leq r\\
        \displaystyle{-\frac{2r+\rho-1}{(d+\rho-1)(\rho-1)}} & 1< j\leq d\\
    \end{cases}
    \end{equation}
    where $\rho = a\beta^{-1}$. We note that under our assumptions $a+(d-1)\beta\not \equiv 0$ and $a-\beta\not\equiv 0$, and so $b_j$ always exists. The rest of the proof follows from the proof in \cite{taylortwographs}; by expanding we get 
    \[a=\ip{\gamma,\gamma}= \sum_{k=1}^db_k\ip{\gamma,\psi_k}=rb_1\beta+(r-d)b_d\beta.\]
    Then plugging in for $b_1$ and $b_d$ we get $4r^2-4dr+(\rho-1)^2(d+\rho)\equiv 0$, of which $r$ and $d-r$ are roots.
\end{proof}

The next result generalizes Theorem 5.6 of \cite{taylortwographs}, whose proof works in the finite field setting as well.
\begin{lem}
    \label{lem:twointersectionnumbers}
    Let $\Phi$ be an $(a,b)$-equiangular system of $n$ lines in a $d$-dimensional orthogonal geometry $V$ such that $a^2\neq b$ and $a\neq 0$.
    Assume there exists $\Gamma\se \Phi$ that is a maximal $\beta$-incoherent set ($|\Gamma|=d$) that is also linearly independent. Then for distinct $\gamma$ and $\delta$ not in $\Gamma$
    \begin{equation} \label{eq:1}
        |\Gamma_1(\gamma)\cap\Gamma_1(\delta)|\equiv |\Gamma_1(\gamma)|-\Delta
    \end{equation}
    where $\Delta$ is either $(\rho-1)^2/4$ or $(\rho^2-1)/4$ and $\rho=a\beta^{-1}$
    Furthermore when $\Char\F>d$ then $|\Gamma_1(\gamma)\cap\Gamma_1(\delta)|$ is the smallest integer which satisfies \eqref{eq:1}.
\end{lem}
\begin{proof}
    Under the same set up as in \cref{lem:rootsofpolyarecardinality} we will fix two distinct elements $\gamma,\delta\not\in\Gamma$ and let $(\psi_j)_{j=1}^d$ be switching equivalent to the vectors of $\Gamma$ such that every pairwise scalar product is $\beta$, and we will assume that the vectors are rearranged such that
    \begin{align*}
    \gamma&=\sum_{1\leq j\leq r}b_1\psi_j+\sum_{r< j\leq d}b_d\psi_j \quad \textrm{and}\\
    \delta&=\sum_{1\leq j\leq s}b_1\psi_j+\sum_{s< j\leq r}b_d\psi_j+\sum_{r< j\leq t}b_1\psi_j+\sum_{t< j\leq d}b_d\psi_j.\end{align*}
    Notice that due to \cref{lem:rootsofpolyarecardinality} we have that $s+(t-r)\equiv r$ as the size of $\Gamma_1(\gamma)$ is equivalent for all $\gamma$, meaning $t\equiv 2r-s$. Similarly, the size of the intersection $|\Gamma_1(\gamma)\cap \Gamma_1(\delta)|=s$
    Notice that using the above basis we can determine that 
    \begin{align*}
        \ip{\gamma,\delta}&=\sum_{1\leq j\leq s}b_1\ip{\gamma, \psi_j}+\sum_{s< j\leq r}b_d\ip{\gamma, \psi_j}+\sum_{r< j\leq t}b_1\ip{\gamma, \psi_j}+\sum_{t< j\leq d}b_d\ip{\gamma, \psi_j}\\
        &=sb_1\beta + (r-s)b_d\beta - (r-s)b_1\beta - (d-(2r-s))b_d\beta\\
        &= (2s-r)b_1\beta + (3r-2s-d)b_d\beta.
    \end{align*}
    Let $\epsilon = \beta^{-1} \ip{\gamma,\delta}=\pm 1$. Expanding this and using \cref{lem:rootsofpolyarecardinality} by
    plugging in for $r^2=dr-\frac{1}{4}(\rho-1)^2(d+\rho)$ and the values for $b_1$ and and $b_j$ from \eqref{eq:gnarlybs}, we get 
    $s\equiv r+\frac{1}{4}(\rho-1)(\epsilon-\rho)$ which proves the statement for each choice of $\epsilon$.
\end{proof}

The following result was shown by \cite{taylortwographs}, and the proof uses only that $\Phi$ forms a regular two-graph and not any properties of the underlying system of lines.

\begin{lem}
    \label{lem:whatisneededbutnotshown}
    Let $\Phi$ be an $(a,b)$-equiangular system of $n$ lines in a $d$-dimensional orthogonal geometry $V$ such that $a^2\neq b$ and $a\neq 0$.
    Assume there exists $\Gamma\se \Phi$ that is a maximal $\beta$-incoherent set ($|\Gamma|=d$) that is also linearly independent. If $\Phi$ forms a regular two-graph with parameters $(n,\ell,m)$ then
    \[\sum_{\gamma\in\Phi-\Gamma}|\Gamma_1(\gamma)||\Gamma_2(\gamma)|=\frac{\ell|\Gamma|(|\Gamma|-1)}{2}.\]
\end{lem}

Now that we have a notion of incoherent vectors that matches the behavior of incoherent real lines we can prove equivalent structural results. The first result is a generalization of Theorem 4.9 in \cite{gillespie-2018-equiangular}. 

\begin{thm}
    \label{thm:youreadesignharry}
    Let $\Phi$ be an $(a,b)$-equiangular system of $n$ lines in a $d$-dimensional orthogonal geometry $V$ over $\F_q$ such that $a^2\neq b$, $a\neq 0$, $\Char\F_q>d$, and $\Phi$ induces a regular two-graph with parameters $(n,\ell,m)$.
    Assume also that there exists $\Gamma\se \Phi$ that is a maximal $\beta$-incoherent set ($|\Gamma|=d$) that is also linearly independent
    where $|\Gamma_1(\gamma)|=g_1$ for all $\gamma$ not in $\Gamma$ and $g_2=|\Gamma|-g_1$.
    
    \noindent If $g_1\neq g_2$ then $(\Gamma, \mathcal B_i)$ is a $2$-$(d,g_i,\lambda_i)$ design for $i=1,2$ such that \[\mathcal B_i=\set{\Gamma_i(\gamma)}{\gamma\in \Phi-\Gamma}\quad \textrm{and} \quad \lambda_i=\frac{\ell(g_i-1)}{2g_j} \enskip \text{for} \enskip j\neq i.\]
    Furthermore, when $n>2d$, then $(\Gamma, \mathcal B_1)$ is a quasi-symmetric $2$-$(d,g_1,\lambda_1;s_1,s_2)$ design, where $s_1$ and $s_2$ are the smallest integers satisfying
    \[s_1\equiv g_1-(\rho-1)^2/4\quad \textrm{and} \quad s_2\equiv g_1-(\rho^2-1)/4,\]
    where $\rho=a\beta^{-1}$. Additionally:
    \begin{itemize}
     \item If $n=d(d+1)/2$ then $(\Gamma, \mathcal B_1)$ is a $4$-design;
     \item If $n=2d$, $(\Gamma, \mathcal B_1)$ is a symmetric $2$-$(d,g_1,\lambda_1;s)$ with $s=s_1$ or $s_2$; and
     \item If $g_1=g_2$ then $(\Gamma, \mathcal B_1\cup \mathcal B_2)$ is a $2$-$(d,d/2,n-d-\ell)$ design.
     \end{itemize}
\end{thm}

The proof of this statement is very similar to the proofs in \cite{gillespie-2018-equiangular}, which relies primarily on the fact that $\Phi$ gives rise to a regular two-graph.

\begin{proof}
    As in \cref{rem:switchit} we will assume that $\Gamma$ has all scalar products equalling $\beta$.
    Suppose $|\Gamma|=g$ and fix $\alpha,\beta\in \Gamma$.
    For all $\gamma$ not in $\Gamma$, where $\{\alpha,\beta,\gamma\}$ is a coherent triple, it would be the case that $\ip{\alpha}{\gamma}=-\ip{\beta}{\gamma}$, so
    $\alpha\in\Gamma_i(\gamma)$ and $\beta\in\Gamma_j(\gamma)$ for $i\neq j$. 
    Likewise for all $\gamma$ not in $\Gamma$ whose addition makes $\{\alpha,\beta,\gamma\}$ an incoherent triple we have that the scalar products would be the same so $\alpha,\beta\in\Gamma_i(\gamma)$ for some $i=1,2$. 
    For each $i=1,2$ we define $k_i$ to be the number of $\gamma$ not in $\Gamma$ which makes $\{\alpha,\beta,\gamma\}$ an incoherent triple such that $\alpha,\beta\in \Gamma_i(\gamma)$. Because $\Phi$ induced a regular two-graph with parameters $(n,\ell,m)$ the number of $\gamma\not\in\Gamma$ which form coherent triples with $\alpha$ and $\beta$ is the parameter $\ell$ and therefore $k_1+k_2=n-g-\ell$. By the properties of regular two-graphs outlined in \cite{gillespie-2018-equiangular}, in particular Equation 4.11, we have in addition that $k_1g_1+k_2g_2=(n-g-\frac{3\ell}{2})g+\ell$.

    Solving for $k_1$ and $k_2$ in the case where $g_1\neq g_2$ and using \cref{lem:whatisneededbutnotshown} gives us that $k_1=\frac{\ell(g_1-1)}{2g_2}$ and $k_2=\frac{\ell(g_2-1)}{2g_1}$.
    Notice that these are independent of the choice of $\alpha$ and $\beta$, meaning any pair of elements in $\Gamma$ are contained in $k_i$ of the blocks $\mathcal B_i=\set{\Gamma_i(\gamma)}{\gamma\not\in\Gamma}$, gives us a $2$-$(g,g_i,k_i)$ design for $i=1,2$. Likewise when $g_1=g_2$, we have that any pair of elements is contained in $k_1+k_2=n-g-\ell$ blocks of $\mathcal B_1\cap \mathcal B_2$, giving a $2$-$(g,g/2,n-g-\ell)$ design.
    
    Assume that $|\Gamma_1(\gamma)|< |\Gamma_2(\gamma)|$, which is the case if and only if $|\Gamma_1(\gamma)|<d/2$. By Fisher's inequality \cref{lem:fish}, because $|\Gamma_1|<d/2$ we have that $|\mathcal B_1|\geq d$.
    We will consider the case where $n=2d$ where $|\mathcal B_1|=n-|\Gamma|= d$, because each block corresponds to a vector $\gamma$ not in $\Gamma$.
    From Theorem 1.15 in \cite{Cameron_Lint_1991}, we know that this must also mean that any two distinct block intersections have the same number of points, making $(\Gamma,\mathcal B_1)$ a symmetric $2$-$(|\Gamma|,|\Gamma_1|,\lambda_1, s)$.
    
    If $n>2d$ and therefore $|\mathcal B_1|> d$, there must be at least two block intersection numbers by Theorem  1.15 in \cite{Cameron_Lint_1991}. From Lemma~\ref{lem:twointersectionnumbers} we know that there can only be two block intersection numbers, making $(\Gamma,\mathcal B_1)$ a quasi-symmetric $2$-$(|\Gamma|,|\Gamma_1|,\lambda_1; s_1, s_2)$ design where $s_1$ and $s_2$ are the smallest integers satisfying the following
    \[s_1\equiv g_1-(\rho-1)^2/4\quad \textrm{and} \quad s_2\equiv g_1-(\rho^2-1)/4,\]
    where $\rho=a\beta^{-1}$ which follows from \cref{lem:twointersectionnumbers}.

    Finally, consider the special case where $n=d(d+1)/2$. We know that $|\mathcal B_1|=n-d=\frac{1}{2}d(d-1)$. Therefore by Proposition 5.6 in \cite{Cameron_Lint_1991} (and also Proposition 13 in \cite{quasi-2-designs}), we have that $(\Gamma, \mathcal B_1)$ is a $4$-design and therefore also a $3$-design.
\end{proof}

\section*{Acknowledgements}
I.\ J.\ would like to thank James Wilson for the numerous helpful books and discussions on non-degenerate Hermitian scalar products. E.\ J.\ K.\ would like to thank Nicolas Bolle and Gene Kopp for engaging conversations on the structure of frames over finite fields during the Summer of Frame Theory 2021.

\bibliographystyle{plain}
\bibliography{StructureFramesFiniteFields.bbl}

\end{document}